\documentclass[11pt]{amsart}
\usepackage {fancybox}
\usepackage{graphicx}
\usepackage{ascmac}
\usepackage{tikz-qtree}
\usepackage{forest}
\usepackage{mathrsfs}
\usepackage{amscd,verbatim}
\usepackage{floatflt}
\usepackage{amssymb}
\usepackage{tikz-cd}
\usepackage[all]{xy}
\usepackage[hdivide={20mm,,20mm}, vdivide={20mm,,20mm}]{geometry}
\usepackage{stmaryrd} % for \mapsfrom
\usepackage[colorlinks,linkcolor=blue,citecolor=blue,urlcolor=red]{hyperref}
\usepackage[normalem]{ulem}

\newcommand{\MpU}{(\ol{U},U^{\infty})}

\newcommand{\A}{\mathbf{A}}
\renewcommand{\AA}{\mathbb{A}}
\newcommand{\mbA}{\mathbb{A}^1}
\newcommand{\mbAm}{\mathbb{A}^m}

\newcommand{\deAm}{(\mathbb{A}^{1},d_{{1}}\{0\}) \otimes (\mathbb{A}^{1},d_{{2}}\{0\}) \otimes \cdots \otimes (\mathbb{A}^{1},\emptyset)}
\newcommand{\deAmo}{(\mathbb{A}^{1},d_{{1}}\{0\}) \otimes (\mathbb{A}^{1},d_{{2}}\{0\}) \otimes \cdots \otimes \{0\}}

\renewcommand{\P}{\mathbf{P}}

\newcommand{\N}{\mathbb{N}}
\newcommand{\PP}{\mathbb{P}}

\newcommand{\Z}{\mathbb{Z}}

\newcommand{\sO}{\mathcal{O}}

\newcommand{\Zt}{\Ztr}

\newcommand{\bZ}{\mathbb{Z}}

\newcommand{\fM}{\mathfrak{M}}

\newcommand{\fZ}{\mathfrak{Z}}

\newcommand{\Cor}{\operatorname{\mathbf{Cor}}}

\newcommand{\HI}{\operatorname{\mathbf{HI}}}

\newcommand{\Rec}{{\operatorname{\mathbf{Rec}}}}

\newcommand{\Cone}{\operatorname{Cone}}

\newcommand{\ul}[1]{{\underline{#1}}}

\newcommand{\DM}{\operatorname{\mathbf{DM}}}

\newcommand{\MDM}{\operatorname{\mathbf{MDM}}}

\newcommand{\ulMDM}{\operatorname{\mathbf{\underline{M}DM}}}

\newcommand{\Coker}{\operatorname{Coker}}

\newcommand{\Spec}{\operatorname{Spec}}

\newcommand{\Sm}{\operatorname{\mathbf{Sm}}}

\newcommand{\rec}{{\operatorname{rec}}}

\newcommand{\tr}{{\operatorname{tr}}}
\newcommand{\proper}{{\operatorname{prop}}}

\newcommand{\eff}{{\operatorname{eff}}}

\newcommand{\red}{{\operatorname{red}}}

\newcommand{\Zar}{{\operatorname{Zar}}}
\newcommand{\Nis}{{\operatorname{Nis}}}
\newcommand{\et}{{\operatorname{\acute{e}t}}}

\newcommand{\id}{{\operatorname{Id}}}

\newcommand{\CH}{{\operatorname{CH}}}

\renewcommand{\lim}{\operatornamewithlimits{\varprojlim}}
\newcommand{\colim}{\operatornamewithlimits{\varinjlim}}

\newcommand{\ol}{\overline}

\renewcommand{\phi}{\varphi}
\renewcommand{\epsilon}{\varepsilon}

\newcommand{\CI}{{\operatorname{\mathbf{CI}}}}

\newcommand{\MDMeff}{\operatorname{\mathbf{MDM}^\eff}}

\newcommand{\bcube}{{\ol{\square}}}
\newcommand{\cube}{\square}

\newcommand{\Zpu}{\Z[1/p]}
\newcommand{\Zputr}{\Z[1/p]_{\text{tr}}}
\newcommand{\lpnbcube}{\Zputr(\bcube^{(l/p^{n})})}
\newcommand{\lpnobcube}{\Zputr(\bcube^{(l/p^{n+1})})}

\newcommand{\M}{\mathbf{M}}

\newcommand{\ulM}{\underline{\M}}

\newcommand{\ulMNST}{\operatorname{\mathbf{\underline{M}NST}}}
\newcommand{\ulMCor}{\operatorname{\mathbf{\underline{M}Cor}}}
\newcommand{\ulomega}{\underline{\omega}}

\newcommand{\MpMnZ}{\ol{M},\Minf + n\ol{Z}}

\def\bZ{\mathbb{Z}}

\def\Ztr{\bZ_\tr}

\def\isom{\overset{\cong}{\longrightarrow}}

\def\Minf{M^{\infty}}

\def\Zinf{Z^\infty}

\newcommand{\MpM}{(\ol{M},\Minf)}
\newcommand{\MpZ}{(\ol{Z},\Zinf)}

\newcounter{spec}
{\end{list}}%

\newtheorem{lemma}{Lemma}[section]
\newtheorem{thm}[lemma]{Theorem}

\newtheorem{prop}[lemma]{Proposition}

\newtheorem{cor}[lemma]{Corollary}

\newtheorem{expe}[lemma]{Expectation}
\theoremstyle{definition}
\newtheorem{situ}[lemma]{Situation}
\newtheorem{defn}[lemma]{Definition}

\theoremstyle{remark}

\newtheorem{remark}[lemma]{Remark}

\numberwithin{equation}{section}

\newcommand{\udo}{\underline{\omega}}

\newtheorem{notation}[lemma]{Notation}

\newcommand{\mf}{\mathfrak}

\begin{document}

\title[Gysin triangles in the category of motifs with modulus]{Gysin triangles in the category of motifs with modulus}
\author{Keiho Matsumoto}

\maketitle

%%%%%%%%%%%%%%%%%%%%% Introduction %%%%%%%%%%%%%%%%%%%%%

\begin{abstract}
%Early 2000s Voevodsky introduces a category of mixed motives $\DM^\eff(k)$ for any field $k$ and prove some fundamental properties on $\DM^\eff$. In particular, he define motivic cohomology via Susulin complex, and prove some comparison theorems between $l$-torsion motivic cohomology groups and $l$-torsion \'etale cohomology groups for a prime $l$ which is invertible in $X$. On the other hand, to study $p$-adic theory, $\DM^\eff$ does not work well. Since the 2010s Bruno, Miyazaki, Saito, Yamazaki began the study about motives theory of modulus pairs. 

In this paper, we study a Gysin triangle in the category of motives with modulus (Theorem \ref{GYSIN}). We can understand this Gysin triangle as a motivic lift of the Gysin triangle of log-crystalline cohomology due to Nakkajima and Shiho. After that we compare motives with modulus and Voevodsky motives (Corollary \ref{pinvert}). 
The corollary implies that an object in $\MDMeff$ decomposes into a $p$-torsion part and a Voevodsky motive part.
%The corollary claims essentially parts of motives with modulus are $p$-torsion and $p$-adic parts. 
We can understand the corollary as a motivic analogue of the relationship between rigid cohomology and log-crystalline cohomology.
%and use it to deduce a structure theorem about %
%$\operatorname{\mathbf{\underline{M}DM}^{eff}}$. %
\end{abstract}

\section{Introduction}
%\subsection{Main Theorem}
The Gysin triangle (see \cite[Prop.3.5.4]{V00b}) in Voevodsky's category of motives $\DM^{\eff}$ is a remarkable result which is a motivic analogue of the purity theorem of \'etale cohomology \cite[3, XVI, Thm.3.7]{SGA4}. In this paper we shall prove a generalization of Voevodsky's theorem in the setting of motives of modulus pairs. %in $\ulMDM^\eff$, 
Our theorem is an analogue of the Gysin triangle of \mbox{(log-)crystalline} cohomology (see \cite[(2.18.8.2)]{NS08}). As a corollary we give a remarkable equivalence which claims that the essential parts of a motive with modulus are the $p$-torsion part and the Voevodsky part. Our proof uses %The goal of this paper is to formulate a Gysin triangle in $\ulMDM^\eff$ and prove it via 
the smooth blow up formula in $\ulMDM^\eff$ (see \cite{KS19}) and a new weighted smooth blow up formula (see Section~\ref{sec:blowupFormula}).

To formulate his Gysin triangle, Voevodsky uses a smooth variety and a smooth closed subvariety. To formulate our Gysin triangle in $\ulMDM^\eff$ we replace the smooth variety by a modulus pair with smooth total space and a modulus whose support is a strict normal crossing divisor, and replace the closed subvariety by a prime smooth Cartier divisor which intersects the modulus properly.

\begin{situ} \label{situ}
Let $\ol{M}$ be a smooth scheme over a field, $M^\infty \subset \ol{M}$ an effective Cartier divisor, $\ol{Z} \subset \ol{M}$ a smooth integral closed subscheme not contained in $M^\infty$ such that the support $|M^{\infty} + \ol{Z}|$ is a strict normal crossings divisor on $\ol{M}$. Write $Z^\infty$ for the intersection product of $M^\infty$ and $\ol{Z}$.%In particular $M^\infty$ also has support a strict normal crossings divisor on $\ol{M}$, the codimension of $\ol{Z}$ in $\ol{M}$ is one, and the intersection $Z^\infty = M^\infty \cdot_{\ol{M}} \ol{Z}$ has support a strict normal crossings divisor on $\ol{Z}$.
\end{situ}

Our main goal is the following two theorem.

\begin{thm}(Tame Gysin triangle) \label{GYSIN}
In the notation of Situation~\ref{situ}, %
%Consider modulus pairs $\MpM, \MpZ \in \ulMCorls$ where $\overline{Z}$ is an integral effective Cartier divisor on $\overline{M}$ not contained in $M^{\infty}$, where $Z^{\infty} = M^{\infty}._{\overline{M}}\overline{Z}$, and we suppose $|M^{\infty} +\ol{Z}|$ is a strict normal crossings divisor on $\ol{M}$. Then
there exsit a distinguish triangle  
%\[
%\ulM(\ol{M},M^\infty + \ol{Z}) \to \ulM(\ol{M},M^\infty) \to Th(N_ZM,1,0) \to \ulM(\ol{M},M^\infty + \ol{Z})[1].
%\]
%where $Th(N_ZM,1,0)$ is the Thom space of weight $(1,0)$, cf.  Definition~\ref{Thom}. Moreover projective bundle formula, cf. Theorem~\ref{projective} provides the following distinguished triangle in $\ulMDMNis$.
\[
\ulM(\ol{M},M^\infty + \ol{Z}) \to \ulM(\ol{M},M^\infty) \to \ulM(\ol{Z},Z^\infty)(1)[2] \to \ulM(\ol{M},M^\infty + \ol{Z})[1],
\]
in $\ulMDM^\eff$.%, where we set $Z^\infty$ to be the intersection product of $\Minf$ and $\ol{Z}$.
\end{thm}

%\begin{remark} \label{rem:TatePBF}
%We remark that a projective bundle formula would imply $Th(N_ZM,1,0)\simeq \ulM(\ol{Z},Z^\infty)(1)[2]$. Kahn-Miyazaki-Saito-Yamazaki are working on this.
%\end{remark}

Theorem~\ref{GYSIN} leads to the following.

\begin{cor}[Theorem \ref{redcase}]\label{1.3}
Let $X$ be a smooth variety over $k$ which has a compactification $\ol{X}$ such that $\ol{X}$ is smooth and $|\ol{X}\backslash X|$ is a strict normal crossings divisor on $\ol{X}$, then the unit
\[
\ulM(\ol{X},|\ol{X}\backslash X|_{\red}) \to \ul{\omega}^{\eff}({\bf{M}}(X))
\]
of the adjunction $\ul{\omega}_{\eff}:\ulMDM^\eff \rightleftarrows\DM^{\eff}:\ul{\omega}^\eff$ is an isomorphism.
\end{cor}
Moreover, as an application of this corollary %
%the concrete symbol of the image of $\ul{\omega}^\eff:\DM^\eff \to \ulMDM^\eff$, 
we get the following equivalence, which philosophically has been expected since the beginning of the theory of motives with modulus. 

\begin{cor}[{Corollary~\ref{Acharap}}]
If the base field $k$ has characteristic $p \geq 2$, for any modulus pair $\MpM$  such that $\ol{M}$ is smooth and $M^\infty_\red$ is strict normal crossing, then there is an isomorphism in $\ulMDM^\eff(k,\Zpu)$
\[
{\ulM\MpM}_{\Zpu} \simeq {\ulM(\ol{M},M^\infty_\red)}_{\Zpu}.
\]
\end{cor}

\begin{defn}
We define $\MDMeff$ as the smallest full triangulated subcategory of $\ulMDM^\eff$ which contains all of proper modulus pairs and is closed under small coproducts. 
\end{defn}

The category $\MDMeff$ is equivalent to the category in \cite[Definition~3.2.4]{KMSY20} because of \cite[Threorem~3.3.1(2),~Theorem~5.2.2]{KMSY20}.

\begin{thm}[Theorem \ref{ppinv}]\label{pinvert}
If the base field $k$ has characteristic $p \geq 2$, admits log resolution of singularities, and $R$ is commutative ring containing $1/p$ then 
\[
\MDMeff(k,R) \simeq \DM^\eff(k,R).
\]
\end{thm}

%More generally, Binda expects that there is an equivalence $\ulMDM^{\eff}(k, \Z[\tfrac{1}{p}]) \cong \logDM^{\eff}(k, \Z[\tfrac{1}{p}])$, where $\logDM^{\eff}$ is the category from \cite{BPO}. Cf. Corollary~\ref{Acharap}. \\

%In section 6, we turn our attention to a construction of realization functor of motives with modulus whchi is never appear in Voevodsky’s motives theory.
%\newcommand{\calO}{\mathcal{O}}
%\newcommand{\tilO}{\tilde{\mathcal{O}}}
%\begin{thm}[Euler realization]\label{coherent}
%There is a Nisnevich sheaf $\tilde{\mathcal{O}}\in Sh_\Nis(\ulMCor)$, which sent a modulus pair $(\ol{M},M^\infty)$ to $\Gamma(\ol{M},\mathcal{O}_{\ol{M}}(M^\infty-|M^\infty|))$and the sheaf $\tilO[0]$ is a cube-local object in $D(PST(\ulMCor)))$, satisfy 
%\[
%\Hom_{\ulMDM^\eff}(\ulM(\ol{M},M^\infty),\tilO[i])\simeq H^i(\ol{M},\calO(M^\infty-|M^\infty|))
%\]
%where the right hand side is %coherent cohomology group.
%\end{thm}

%The section 6 seems to be independent from other results, but the Theorem \ref{coherent} is an evidence of Conjecture \ref{chzero} (see subsection 1.4) which is an characteristic zero analogue of Thereom \ref{pinvert}. We hope that Theorem \ref{coherent} can be understood as a relationship between motives with modulus theory and birational geometry. The result is hochschilt homology 

Let us discuss the relationship between our results and other work.

\subsection{Relationship to the Gysin triangle for (log-)Crystalline cohomology} %Now we discuss an analogue between our Gysin triangle in $\ulMDM$ and a Gysin triangle for crystalline cohomology. 
%We expect that $\ulMDMNis$ or $\ulMDM_{\proper}^\eff$ has a log-crystalline realization, which takes our Gysin triangle to the crystalline one. %
First lets state the Gysin triangle for crystalline cohomology, and a comparison theorem between rigid cohomology and crystalline cohomology.
%In \cite{V00b}, Voevodsky prove the following Gysin triangle in $\DM^\eff$.
%\begin{prop}\cite[Proposition 3.5.4]{V00b}
%Let $X$ be a smooth scheme over $k$ and $Z$ be a smooth Cartier divisor in $X$. Then there is a canonical distinguished triangle in $\DM^{\text{eff}}$ of the form
%\[
% M(X \backslash Z) \to M(X) \to M(Z)(1)[2] \to M(X \backslash Z)[1].
%\]
%\end{prop}

%The proposition give us a motivic aspect of purity theorem of $l$-torsion \'etale cohomology (see SGA). On the other hand, Shiho prove a purity theorem of crystalline cohomology.

\begin{thm}[{\cite[Eq.2.18.8.2]{NS08}, \cite{Shi02}}] \label{Shiho} Let $W$ be the Witt ring of the base field, let $K$ be the fractional field of $W$, and set $S=\Spec W$. Consider the push forward functors $f_{-\slash S}: Sh(-/S)_{\text{crys}}\to Sh_\Zar(S)$ from the (log-)crystalline sites of log schemes over $S$ to the Zariski site of $S$, and the structure sheaves $\mathcal{O}_{-/S}$ on $(-/S)_{\text{crys}}$. In the notation of Situation~\ref{situ}, there is a long exact sequence of Zariski sheaves on $S$:
%\[
%\cdots  \rightarrow R^{n-2}f_{\ol{Z}/S}(\mathcal{O}_{\ol{Z}\slash W})(-1)\rightarrow R^{n}f_{\ol{M}/S}(\mathcal{O}_{\ol{M}/S})\rightarrow R^{n}f_{(\ol{M},\ol{Z}/S)}(\mathcal{O}_{(\ol{M},\ol{Z})/S})
% & \rightarrow R^{n-1}f_{\ol{Z}/S}(\mathcal{O}_{\ol{Z}/S})(-1)\rightarrow\cdots
%\]
\begin{align*}
\cdots 
\rightarrow R^{n-2}f_{\ol{Z}/S}%
&(\mathcal{O}_{\ol{Z}\slash W})(-1)
\rightarrow R^{n}f_{\ol{M}/S}(\mathcal{O}_{\ol{M}/S}) \\
&\rightarrow R^{n}f_{(\ol{M},\ol{Z}/S)}(\mathcal{O}_{(\ol{M},\ol{Z})/S})
\rightarrow R^{n-1}f_{\ol{Z}/S}(\mathcal{O}_{\ol{Z}/S})(-1)
\rightarrow\cdots
 \end{align*}
 and we have a natural and functorial isomorphism
\[
comp:  H^i_{\text{crys}}((\ol{M},\ol{Z})\slash W)\otimes_{W}K
\simeq 
H^i_{\text{rig}}(\ol{M}\backslash \ol{Z}\slash K).
 \]
\end{thm}

%Let us explain relationship between our results and Theorem~\ref{Shiho}. 
\begin{expe} \label{expe:crysReal}
We expect that there exists an exact ``crystalline realization functor'' $$\mathbb{R}\Gamma_{crys}:\MDMeff(k,W) \to D(W)$$ satisfying %
%Brown representability, 
\[ \mathbb{R}\Gamma_{crys}\bigl(\ulM(\ol{M},\emptyset)\bigr) \simeq \mathbb{R}\Gamma\bigl(S,\mathbb{R}f_{\ol{M}/S}(\mathcal{O}_{\ol{M}/S})\bigr)  \textrm{ and } 
\]
\[
\mathbb{R}\Gamma_{crys}\bigl(\ulM(\ol{M},\ol{Z})\bigr) \simeq \mathbb{R}\Gamma\bigl(S,\mathbb{R}f_{(\ol{M},\ol{Z})/S}(\mathcal{O}_{(\ol{M},\ol{Z})/S})\bigr). \] %
\end{expe}

%\[ \DM^\eff(k, W)[\tfrac{1}{p}], \qquad  \]
%sends ${\bf{M}}(\ol{M} \backslash M^\infty)[\tfrac{1}{p}]$ to $\ulM(\ol{M},M^\infty_\red)[\tfrac{1}{p}]$

In this case, the tame Gysin triangle Theorem~\ref{GYSIN} would be a motivic lifting of the first claim of Theorem~\ref{Shiho}. %

Now consider rigid cohomology. Milne-Ramachandran \cite{MR09} construct\footnote{This can be constructed as follows. Since $K$ contains $\mathbb{Q}$, by Ayoub's work \cite[App.B]{A14} there is an equivalence $\DM^\eff(k,K) \cong \textbf{DA}^\eff_{\et}(k,K)$ and so it suffices to construct a functor $\textbf{DA}^\eff_{\et}(k,K) \to D(K)$. Since rigid cohomology satisfies \'etale descent and $\AA^1$-homotopy invariance (see \cite{CT03}), the factorization of Besser’s rigid complex $\mathbb{R}\Gamma_{rig}(-):(\Sm\slash k) \to D(K)$ (see \cite[4.9, 4.13]{B00}) through $\textbf{DA}^\eff_{\et}(k,K)$, is a rigid realization.} a rigid realization 
\[ \mathbb{R}\Gamma_{rig} :\DM^\eff(k,K) \to D(K) \]
satisfying
\[ \mathbb{R}\Gamma_{rig}(M(X)) = \mathbb{R}\Gamma_{rig}(X) \]
for $X$ smooth where the right hand side is Besser’s rigid complex. By Corollary \ref{pinvert} the functor ${\ul{\omega}^\eff}_K:\DM^\eff(k,K)\to \ulMDM^\eff_{\proper}(k,K) $ is an equivalence, with quasi-inverse $\ul{\omega}_{\eff, K}$. Since $\ul{\omega}_{\eff}$ sends $\ulM(\ol{M},M^\infty_\red)$ to ${\bf{M}}(\ol{M} \backslash M^\infty)$, the second claim of Theorem~\ref{Shiho} produces a natural isomorphism of functors 
%\[ 
%\xymatrix{
%\ulMDM^\eff_{\proper}(k,K) \ar[dr]^{\mathbb{R}\Gamma_{crys} \otimes_W K} \ar[d]_{\ul{\omega}_*} &  \\
%\DM^\eff(k,K) \ar[d]|\hole|\cong & D(K) \\
%\textbf{DA}^\eff_{\et}(k,K) \ar[ur]_{\mathbb{R}\Gamma_{rig}} & 
%}
%\] 
\[ 
\xymatrix{
\MDMeff(k,K) \ar[dr]\ar@{}@<-0.7ex>[dr]_{\mathbb{R}\Gamma_{crys} \otimes_W K} \ar[rr]^-{\ul{\omega}_{\eff}} & \ar@{}[d]|(0.65)\cong|(0.45){comp} & 
\DM^\eff(k,K)  \ar[dl]^{\mathbb{R}\Gamma_{rig}}\\
& D(K) & \\
}
\]

In this light, if Expectation~\ref{expe:crysReal} holds, then the equivalence $\ul{\omega}^\eff$ of Corollary~\ref{pinvert} will be a motivic lifting of the isomorphism $comp$ of Theorem~\ref{Shiho}.
%$\mathbb{R}\Gamma_{crys} \otimes_W K \simeq \mathbb{R}\Gamma_{rig} \circ a_{\et} \circ {\ul{\omega}_*}_K$ where $a_\et$ is the \'etale sheafification $\textbf{DA}^\eff(k,K) \to \textbf{DA}_\et^\eff(k,K)$, conversely Corollary \ref{pinvert} is a motivic aspect of the second claim of the Theorem. 

\begin{remark}
Binda-Park-Østvær constructed in \cite[Section 1.3.2]{BPO} a framework which is analogous to $\ulMDM^{\eff}$ called \emph{log motives}, and they are pursuing the construction of a log crystalline realisation functor in their framework. It would be very interesting to investigate the relationship between log motives and motives with modulus in the future.
%At present the relationship between the two theories is unclear, 
\end{remark}

\subsection{Relationship to Miyazaki's works on higher Chow groups with modulus} \label{sec:miya}
In \cite{BS14} Binda-Saito define higher Chow groups with modulus generalizing additive higher Chow groups (see \cite{BE03}). Miyazaki proves that after inverting $p$, higher Chow groups with modulus become independent of the modulus.
\begin{thm}\cite[Theorem 5.1]{Miy19b} If base field has characteristic $p$, then for any modulus pair $(\ol{M},M^\infty)$, we have an isomorphism \[\CH^i(\ol{M}|M^\infty,j,\Zpu) \simeq \CH^i(\ol{M}|M_\red^\infty,j,\Zpu).\]
\end{thm}

On the other hand, it is expected that Voevodsky's isomorphism \cite[Cor.4.2.9]{V00b}
\begin{equation}
\CH^{n{-}i}(X, j{-}2i) \cong \hom_{\DM^\eff}(\Z(i)[j], M^c_{gm}(X)) 
\end{equation}
can be generalized to a relationship between higher Chow groups with modulus and $\ulMDM^\eff$. If this is the case, then the equivalence of Corollary~\ref{pinvert} can be seen as an analogue of Miyazaki's independence result.

\subsection{Relationship to reciprocity sheaves} \label{sec:BCKS} If the base field $k$ has characteristic $p$, then for a Nisnevich reciprocity sheaf $F$ (see \cite{KSY14}), the kernel of the canonical surjective morphism $F \to \text{H}_0(F)$ must be $p$-primary torsion, (see \cite[Corollary 3.10]{BCKS17}). In fact %
%Binda, Cao, Kai and Sugiyama %
Binda-Cao-Kai-Sugiyama %
prove an equivalence of categories $\Rec_\Nis[\frac{1}{p}] \simeq \HI_\Nis[\frac{1}{p}]$ between the category of reciprocity sheaves and the category of homotopy invariant Nisnevich sheaves with transfers.

On the other hand, there is a tower of fully faithful functors $\HI_\Nis \overset{i_\rec^\Nis}{\hookrightarrow} \Rec_\Nis  \overset{\omega^\CI_\rec}{\hookrightarrow} \CI_\Nis^{sp}$ (see \cite[Thm.3.6.6]{KSY17}, \cite[Cor.3.8.2]{KSY17}, \cite{BS19}). In analogy to the fact that the heart of $\DM^\eff$ is $\HI_\Nis$ (see \cite[Thm.3.1.12]{V00b}), it is expected that the heart of $\MDMeff$ is $\CI_\Nis^{sp}$ (log version of this story is proved by Binda-Merici \cite[Theorem~5.7]{BM}). By definition the composition $\omega^\CI_\rec \circ i_\rec^\Nis$ is compatible with $\omega^\eff:\DM^\eff \to \MDMeff$. If we assume that the heart of $\MDMeff$ is $\CI_\Nis^{sp}$ then Cor.\ref{pinvert} implies an equivalence $\CI_\Nis^{sp}[\frac{1}{p}] \cong \HI_\Nis[\frac{1}{p}]$. Then the two inclusions $i_\rec^\Nis[\frac{1}{p}]$ and $\omega^\CI_\rec[\frac{1}{p}]$ become equivalences, so Corollary~\ref{pinvert} can be seen as an analogue of this story.

\section*{Acknowledgement}

We thank an anonymous referee for many suggestions which improved the readability of this paper. 

Next, I would like to thank those members of motives with modulus school especially Federico Binda, Shane Kelly, Hiroyasu Miyazaki, and Shuji Saito. 

I must also thank Jun Koizumi, a master student in Tokyo university. He pointed out a mistake in the proof of the open Gysin triangles \cite{Koizumi}.

%\subsection{Remark}
%Let us comment on my paper. We use two triangulated categories $\ulMDM^\eff$ and $\ulMDMNis$, but actually these two should be equivalence (in progress by Kahn-Miyazaki-Saito-Yamazaki). 

%{\color{red}{\subsection{(In)Dependency of our results on \cite{KSY15} and \cite{KSY17}} Let us give some comments on the dependency of our results on \cite{KSY15}. Theorem \ref{GYSIN} is independent of \cite{KSY15} since we use the Thom space instead by Tate motive (cf. Remark~\ref{rem:TatePBF}). On the other hand Corollary~\ref{1.3} and Corollary~\ref{pinvert} are dependent on \cite[Theorem 7.3.1]{KSY15}. Precisely, we these use the claim that for a smooth proper variety $\ol{X}$, the unit $\ulM(\ol{X},\emptyset) \to \ul{\omega}^\eff(M(\ol{X}))$ is an isomorphism.}}
%Recall that $\MDM^{\eff}$ is a version of $\ulMDM^{\eff}$ constructed using modulus pairs whose total space is proper, cf.\cite[Thm.1]{KM18}, %$\MDM^{\eff}$ is the localising subcategory of $\ulMDM^{\eff}$ generated by motives of modulus pairs whose total space is proper, and that 

%%%%%%%%%%%%%%%%%%%%% Bloch $$$$$$$$$$$$$$$$$$$$$$$$$$$$

\newcommand{\OL}{\overline}
\section{Definition and Preparation}

%In this paper, we work over a perfect field $k$. We denote by $\MSmnc$ (resp. $\MSmncf$, $\MCornc$, $\MCorncf$) the full subcategory of $\MSm$ (resp. $\MSmf$, $\MCor$, $\MCorf$) whose objects are modulus pairs $\MpM$ such that $\OL{M}$ is smooth over $k$ and the support $|M^{\infty}|$ of the effective Cartier divisor $M^{\infty}$ is a strict normal crossings divisor on $\OL{M}$, cf. \cite[Definitions 1.3.1, 1.3.2, 1.10.1]{KSY15}.

In this paper, we work over a perfect field $k$. As in \cite[Definition 1.3.1]{KMSY19a} we write $\ulMCor$ for the category of modulus pairs and left proper admissible correspondences. %Following \cite[Definition 1.11]{Sai17} we write $\ulMCorls$ for the full subcategory of $\ulMCor$ whose objects are modulus pairs $\MpM$ such that $\OL{M}$ is smooth over $k$ and the support $|M^{\infty}|$ is a strict normal crossings divisor on $\OL{M}$.
%
%We consider the following diagram
%\[\xymatrix{
%\MSmncf \ar[r]^{b_a} \ar@{^{(}-_>}[d] & \MSmnc \ar[r]^{c} \ar@{^{(}-_>}[d] & \MCornc \ar@{^{(}-_>}[d] & \MCorncf \ar[l]_{b} \ar@{^{(}-_>}[d] \\
%\MSmf \ar[r]^{}  & \MSm \ar[r]^{}  & \MCor & \MCorf \ar[l]_{} \\
%}
%\]
%where each vertical map is a full embedding, and under resolution of singularities they are equivalences, \cite[Corollary 1.10.5]{KSY15}.
%
%The canonical inclusion
%\[ \ulMCorls \hookrightarrow \ulMCor \]
%is a full embedding, and under resolution of singularities is an equivalence, because blowups inside the divisor are isomorphisms in $\ulMCor$, cf.~\cite[Corollary 1.10.5]{KSY15}.
%
We write
%\[
%\Zf:\MSmncf \to PSh(\MSmncf)
%\]
%\[
%\Z:\MSmnc \to PSh(\MSmnc)
%\]
%\[
%\Ztf:\MCorncf \to PSh(\MCorncf)
%\]
%\[
%\Zt:\MCorls \to PSh(\MCorls)
%\]
\[
\Ztr:\ulMCor \to PSh(\ulMCor)
\]
for the associated representable additive presheaf functor.

%The functor $b$ induces an exact functor. 
%\[
%b_{!}:{\bf Sh}_{\Nis}(\MCorncf) \xrightarrow{} {\bf Sh}_{\Nis}(\MCornc) 
%\]
%and for $M \in \MSmncf$, we have $b_{!}\Zf(M) \simeq \Z(M)$, \cite[Proposition 2.5.1]{KSY15}. 

%We define $\ulMDM^{\text{eff}}$ to be the Verdier quotient of ${\bf{D}}(PSh(\ulMCor))$ by the smallest localising subcategory containing all complexes of the form:\\
%(CI) for ${\mathfrak{M}}\in\ulMCor$ ,
%\[
%\Ztr(\mathfrak{M}\otimes \cube) \to \Ztr(\mathfrak{M});
%\]
%(MV) for $\mathfrak{M}\in\ulMCor$ and an elementary Nisnevich cover%
%\footnote{By elementary Nisnevich cover we mean morphisms $\{(\OL{U}, U^\infty) \to \mathfrak{M}, (\OL{V}, V^\infty) \to \mathfrak{M}\}$ such that $\{\OL{U} \to \OL{M}, \OL{V} \to \OL{M}\}$ is an elementary Nisnevich cover in Voevodsky's sense, and $U^\infty, V^\infty$ are the pullbacks of $M^\infty$. By $\mathfrak{U} \times_{\mathfrak{M}} \mathfrak{V}$ we mean $(\OL{U} \times_{\OL{M}} \OL{V}, pr_1^{-1}U^\infty + pr_2^{-1}V^\infty)$.}%
%$(\mathfrak{U},\mathfrak{V})$ of $\mathfrak{M}$, 
%\[
%\Ztr(\mathfrak{U}\times_{\mathfrak{M}}\mathfrak{V}) \to \Ztr(\mathfrak{U}) \oplus \Ztr(\mathfrak{V}) \to %\Ztr(\mathfrak{M}).
%\]

We set $\ulMNST$ to be the category of Nisnevich sheaves on $\ulMCor$ defined in \cite[Definition 4.5.2]{KMSY19a}. 

We define $\ulMDM^\eff$ to be the Verdier quotient
of ${\bf{D}}(\ulMNST)$ by the smallest localising subcategory containing all complexes of the form:\\
(CI) for ${\mathfrak{M}}\in\ulMCor$ ,
\[
\Ztr(\mathfrak{M}\otimes \cube) \to \Ztr(\mathfrak{M}).
\]
Note that complexes of the following form are quasi-isomorphic to zero in $D(\ulMNST)$:\\
(MV) for $\mathfrak{M}\in\ulMCor$ and an elementary Nisnevich cover%
\footnote{By elementary Nisnevich cover we mean morphisms $\{(\OL{U}, U^\infty) \to \mathfrak{M}, (\OL{V}, V^\infty) \to \mathfrak{M}\}$ such that $\{\OL{U} \to \OL{M}, \OL{V} \to \OL{M}\}$ is an elementary Nisnevich cover in Voevodsky's sense, and $U^\infty, V^\infty$ are the pullbacks of $M^\infty$. By $\mathfrak{U} \times_{\mathfrak{M}} \mathfrak{V}$ we mean $(\OL{U} \times_{\OL{M}} \OL{V}, pr_1^{-1}U^\infty + pr_2^{-1}V^\infty)$.}%
$(\mathfrak{U},\mathfrak{V})$ of $\mathfrak{M}$, 
\[
\Ztr(\mathfrak{U}\times_{\mathfrak{M}}\mathfrak{V}) \to \Ztr(\mathfrak{U}) \oplus \Ztr(\mathfrak{V}) \to \Ztr(\mathfrak{M}).
\]

%The Nisnevich sheafification functor $\ul{a}_\Nis:PSh(\ulMCor) \to \ulMNST$ is exact (see Lemma \ref{exatexat}) and has fully
%faithful right adjoint functor $\ul{i}$ (see Theorem \ref{avsi}), the derived functor $D(\ul{a}_\Nis):D(PSh(\ulMCor)) \to D(\ulMNST)$ is a localization. By \cite[Theorem 4.5.7]{KMSY19a}, $D(\ul{a}_\Nis)$ sends complexes (MV) to $0$ in $D(\ulMNST)$, we obtain the following natural localization functors 
%\[
%D(PSh(\ulMCor))\slash \text{(MV)} \to D(\ulMNST) \text{ and }
%\]
%\[
%\ulMDM^\eff \to \ulMDMNis.
%\]
%We expect that $\ulMDM^\eff \simeq \ulMDMNis$. 
We define $\MDMeff$ to be the smallest subcategory of $\ulMDM^\eff$ containing the objects $\ulM(\ol{M},\Minf)$ for modulus pair $(\ol{M}, \Minf)$ such that $\ol{M}$ is proper, and closed under isomorphisms, direct sums, shifts, and cones.

%\begin{rem}
%If k admits resolution of singularities then $\MDMnceff \simeq \MDMne$.
%\end{rem}
We have a functor %$\udo$
\[
\udo:\ulMCor \to \Cor
\]
with $\udo\MpM = M^{\circ}:=\OL{M} ~\backslash ~ M^{\infty}$. This functor $\udo$ induces a triangulated functor
\[
\underline{\omega}_{\text{eff}}: \ulMDM^\eff \to \DM^{\eff}.
\]

\begin{defn}\label{Thom}
In Situation~\ref{situ}, %for $n > m \geq 0$ 
we define the closed Thom space as
\begin{align*}
Th(N_ZM,cl) := \Cone\biggl (
&\ulM(\mathbb{P}(N_{\ol{Z}}\ol{M} \oplus \mathcal{O}), \pi^* Z^\infty + \{\infty\}_{\ol{Z}})  \\
\to &\ulM(\mathbb{P}(N_{\ol{Z}}\ol{M} \oplus \mathcal{O}), \pi^* Z^\infty) \biggr ).
\end{align*}
in $\ulMDM^\eff$, where $\pi: \mathbb{P}(N_{\ol{Z}}\ol{M} \oplus \mathcal{O}) \to \ol{Z}$ is the canonical projection. %
%
%In the case $n, m = 1, 0$ we write $Th(N_Z M,1,0) = Th(N_Z M,cl)$. 
\end{defn}

Notice that the closed Thom space is a lifting of Voevodsky's Thom spaces in the sense that $\ulomega_{\eff}$ sends $Th(N_Z M, cl)$ to $Th(N_{Z^{\circ}} M^\circ)$.

For a smooth variety and  a vector bundle $E$ on $X$, Voevodsky defined the Thom space in $\DM^\eff$:
\[
Th_X(E)=\Cone(\PP_X(E\oplus \sO)\backslash \{\infty\}_X \to \PP_X(E\oplus \sO)).
\]

\begin{remark}
Note that $Th(N_Z M)$ is a direct summand of $\ulM(\mathbb{P}(N_{\ol{Z}}\ol{M} \oplus \mathcal{O}), \pi^* Z^\infty)$ since 
$\ulM(\mathbb{P}(\mathcal{O}), \pi^* Z^\infty)
\simeq \ulM(\ol{Z},Z^\infty)$. Cf.~\cite[Lemma~6]{KS19}. In fact, by the projective bundle formula ~\cite[Theorem 7.3.2]{KMSY20}, the closed Thom spaces %of weight $(1, 0)$ 
are just Tate twists: $Th(N_ZM, cl) \cong \ulM(\ol{Z}, Z^\infty)(1)[2]$.
\end{remark}

\begin{remark}\label{birational}
For any proper birational morphism of schemes $f:X \to Y$ and effective Cartier divisors $Y^\infty \subset Y$ and $X^\infty=f^*Y^\infty$  satisfying $Y \backslash Y^\infty \simeq X \backslash X^\infty$, there is an isomorphism 
\[
\ulM(Y,Y^\infty) \simeq \ulM(X,X^\infty)
\]
in $\ulMCor$. Cf. \cite[Proposition~1.9.2.(b)]{KMSY19a}.
\end{remark}

The following basic homological algebra result will be useful.

\begin{lemma}\label{nine}
Consider a commutative diagram.
\[\xymatrix{
& A \ar[r]^{}  \ar[d] &  B  \ar[d] \ar[r] &  C \ar[d]  \\
 &  A' \ar[d] \ar[r]  &  B' \ar[d] \ar[r]  & C' \ar[d]  &   \\
 & A'' \ar[r]  &  B'' \ar[r] &  C''   &  \\
}
\]
in an additive category $\mathcal{A}$ such that all horizontal and vertical compositions
are zero. Suppose we have a triangulated functor $\Phi:K^b(\mathcal{A}) \to T$ to some
triangulated category $T$ such that (the complexes associated to) all three columns and two of the rows are sent to zero in $T$. Then the (the complex associated to) the third row is sent to zero in $T$ as well.
\end{lemma}

\begin{proof}
Clear.
%First note that given a complex $[X \to Y \to Z]$ in $K^b(\mathcal{A})$ concentrated
%in three degrees,$\Phi[X \to Y \to Z]$ is zero if and only if $\Phi[X \to Y ] \to \Phi Z$ is an isomorphism. So we are trying to show that $\Phi[A'' \to B''] \to \Phi[C'']$is an
%isomorphism. Since all columns are sent to zero, this candidate isomorphism is
%isomorphic to $\Phi[A \to A' \oplus B \to B' ] \to \Phi[C \to C']$. This latter is an isomorphism in $T$, if and only if its cone
%$\Phi[A \to A'
%\oplus B \to B'
%\oplus C \to C'
%]$ is zero. But this latter is the cone of $\Phi[A \to B \to C] \to \Phi[A' \to B' \to C']$ whose source and target are both sent to zero by hypothesis.
\end{proof}

%%%%%%%%%%%%%%%%%%%%%%%%%%%%%%%%%%%%%%%%%%%%%%%%%%%%%%%%%%%%

\section{Excision} \label{sec:excision}

In this section we prove some basic excision results and prove that Thom spaces are invariant under change of étale neighbourhood.

Let $M=\MpM$ and $Z=\MpZ$ be as in Situation~\ref{situ}. For $n\in\Z_{\geq 0}$ we define a presheaf on $\ulMCor$ 
\[
C_{nZ}^{M} = \Coker\bigl(\Ztr\MpU \hookrightarrow \Ztr(\MpMnZ)\bigr)
\]
where $\OL{U} = \OL{M} \setminus \OL{Z}$, $U^\infty = M^\infty|_{\OL{U}}$ and $\Ztr\MpU \to \Ztr(\MpMnZ)$ is induced by the open immersion $\ol{U} \to \ol{M}$.

For a morphism $f:(\ol{M},\Minf) \to (\ol{N},N^\infty)$ induced by a morphism of schemes $\ol{f}:\ol{M} \to \ol{N}$, we call $f$  \emph{minimal} if we have $M^\infty=\ol{f}^*N^\infty$.
\newcommand{\MpN}{(\ol{N},N^{\infty})}
\newcommand{\MpNnZ}{(\ol{N},N^{\infty}+ nf^{-1}\ol{Z})}
\begin{prop} \label{Fetexcion}
Let $f:\MpN \to \MpM$ be an \'etale morphism (i.e., $f$ is induced by an \'etale morphism $\ol{f}:\ol{N} \to \ol{M}$ and is minimal). If $\ol{f}^{-1}\ol{Z} \to \ol{Z}$ is an isomorphism, then for any $n\in\Z_{\geq 0}$, the natural morphism $C_{nf^{-1}Z}^{N} \to C_{nZ}^{M}$ is a isomorphism in $\ulMDM^\eff$.
\end{prop}
\begin{proof} Let $\ol{V}=\ol{N} \backslash f^{-1}\ol{Z}$. We have a diagram in $PSh(\ulMCor)$.\scriptsize
\[\xymatrix{
0 \ar[r] & \Ztr(\ol{V},\ol{V}\cap M^\infty) \ar[r] \ar[d] &  \Ztr\MpNnZ  \ar[r] \ar[d]  & C_{nf^{-1}Z}^N  \ar[r] \ar[d]  &  0  \\
0 \ar[r] &  \Ztr(\ol{U},U^\infty) \ar[r]  &  \Ztr(\MpMnZ)  \ar[r]  & C_{nZ}^M \ar[r]   &  0 \\
}
\]
\normalsize
The left hand-side square is homotopy Cartesian in $\ulMDM^\eff$ by the definition of $\ulMDM^\eff$. So we get the claim. 
\end{proof}
\begin{thm}\label{PEPEPE}
Let $f:\MpN \to \MpM$ be an \'etale morphism. If $\ol{f}^{-1}\ol{Z} \to \ol{Z}$ is an isomorphism, then for any $n \geq m \geq 0$ there is a diagram in $PSh(\ulMCor)$,
\[\xymatrix{
0 \ar[r] & \Ztr\MpNnZ \ar[r]^-{i_N}  \ar[d] &  \Ztr(\ol{N},N^\infty+m f^{-1}\ol{Z})  \ar[r] \ar[d]  & \text{Coker}(i_N)  \ar[r] \ar[d]  &  0  \\
0 \ar[r] &  \Ztr(\MpMnZ) \ar[r]^-{i_M}  &  \Ztr(\ol{M},M^\infty + m\ol{Z})  \ar[r]  & \text{Coker}(i_M)  \ar[r]   &  0 \\
}
\]
such that $\text{Coker}(i_N) \to \text{Coker}(i_M)$ is an isomorphism in $\ulMDM^\eff$.
\end{thm}
\begin{proof}
We consider the following commutative diagram in $PSh(\ulMCor)$,\scriptsize
\[\xymatrix{
&0 \ar[d] &0 \ar[d] &0 \ar[d] &\\ 
0 \ar[r] & \Ztr\MpU \ar[r]^{=}  \ar[d] &  \Ztr\MpU  \ar[d] \ar[r] &  0 \ar[d]  \\
0 \ar[r] &  \Ztr(\MpMnZ) \ar[d] \ar[r]^{~~~i_M}  &  \Ztr(\ol{M},M^\infty+m\ol{Z}) \ar[d] \ar[r]  & \text{Coker}(i_M) \ar[d]^{||} \ar[r] & 0  \\
0 \ar[r] & C_{nZ}^{M} \ar[r]^{c_M} \ar[d] &  C_{mZ}^{M} \ar[r] \ar[d]&  \text{Coker}(i_M) \ar[r] \ar[d] & 0 \\
&0&0&0&
}
\]\normalsize
where $i_M$ is the natural map and $c_M$ is the unique map determined by $i_M$. Now all columns and the two top rows are exact. Now by the nine lemma, we get that the bottom row is also exact. The morphisms $i_M$ and $\ol{f}$ induce the commutative diagram:
\[\xymatrix{
C_{nf^{-1}Z}^{N} \ar[d] \ar[r]^{c_N} & C_{mf^{-1}Z}^{N} \ar[d]\\
C_{nZ}^{M} \ar[r]^{c_M} & C_{mZ}^{M}
}
\]
By Proposition~\ref{Fetexcion}, the vertical morphisms become isomorphisms in $\ulMDM^{\eff}$. Hence the map between the cokernels of the two horizontal presheaf monomorphims become isomorphisms in $\ulMDM^{\eff}$.  
\end{proof}

%\begin{cor}
%$\beta_{n(\{0\}, \varnothing)}(\AA^1, varnothing) / \AA^1, n \{0\}) \to Th(***)$ is an isomorphism.
%\end{cor}

\begin{cor}\label{PEPP}
In the situation of Theorem \ref{PEPEPE} Thom spaces are isomorphic
\[
Th(N_{f^{-1}\ol{Z}}N,cl) \simeq Th(N_{\ol{Z}}M,cl) 
\]
\end{cor}

\begin{proof}
In the situation of Theorem \ref{PEPEPE} for any $n\in \{0,1\}$ the natural morphism
\[\ol{f}:(\mathbb{P}(N_{f^{-1}\ol{Z}}\ol{N} \oplus \mathcal{O}), \pi'^* Z'^\infty +n\{\infty\}_{f^{-1}\ol{Z}})\to(\mathbb{P}(N_{\ol{Z}}\ol{M} \oplus \mathcal{O}), \pi^* Z^\infty +n\{\infty\}_{\ol{Z}})\]
is minimal \'etale morphism where $Z'^{\infty}:=f^{-1}\ol{Z}.N^{\infty}$ and $\pi'$ is the projection $\mathbb{P}(N_{f^{-1}\ol{Z}}\ol{N} \oplus \mathcal{O}) \to f^{-1}\ol{Z}$. Moreover $\ol{f}$ induces an isomorphism $\{\infty\}_{f^{-1}\ol{Z}} \simeq \{\infty\}_{\ol{Z}}$ since $f^{-1}\ol{Z} \simeq \ol{Z}$. By Proposition~\ref{Fetexcion} and Theorem \ref{PEPEPE} we obtain the claim.
\end{proof}

\section{Blow up formula with weight} \label{sec:blowupFormula}

\subsection{Introduction}

Kelly-Saito proved a blow up formula for motives with modulus (see \cite{KS19}), but to construct tame Gysin map we need another formula, namely, Theorem~\ref{Blowupup}. In this section, we calculate some motives of Fano surfaces with modulus, after that we have constructed the formula which we need. We begin with the notation that we will need to perform the deformation to the normal cone technique.

\begin{notation}\label{otimessitu}
In the Situation \ref{situ}, we use the following notations.
\begin{align*}
M &:= (\ol{M},M^\infty),\\
Z &:= (\ol{Z},Z^\infty), \\
\ol{B}_M^{(\ol{Z})}\xrightarrow{\pi_M}\ol{M}\times\PP^1 &: \text{the blow up of } \ol{M}\times\PP^1 \text{ at  }\ol{Z}\times \{0\}, \\
B_M^\infty &:= \pi_M^*(M^\infty\times \PP^1+\ol{M}\times \{\infty\}), \\
W_M &: \text{the strict transform of }\ol{Z}\times\PP^1 \text{ w.r.t. }\ol{B}_M \to \ol{M}\times\PP^1, \\
B_{M,cl}^{(\ol{Z})} &:= (\ol{B}_M^{(\ol{Z})},B_M^\infty + W_M),\\
\ol{U}_M &:= \ol{M}\times\PP^1 ~ \backslash ~{\ol{Z}\times \PP^1}, \\
\ol{E}_M^{(\ol{Z})} &:= \text{the exceptional divisor of }\pi_M, \\
%\bigl(\ol{E}_M \backslash (\ol{E}_M\cap W_M),(\ol{E}_M\cap B_M^\infty) \backslash (W_M\cap B_M^\infty)\bigr),\\
E_{M,cl}^{(\ol{Z})} &:=\bigl(\ol{E}_M,(\ol{E}_M\cap B_M^\infty) + (\ol{E}_M\cap W_M)\bigr).
\end{align*}

%\centering
%\includegraphics[width=6cm]{page1.jpg}

%\[
%M:=(\ol{M},M^\infty),~ Z:=(\ol{Z},Z^\infty),~\ol{B}_M^{(\ol{Z})}\xrightarrow{\pi_M}\ol{M}\times\PP^1 :\text{the blow up of } \ol{M}\times\PP^1 \text{ with  }\ol{Z}\times \{0\},
%\]
%\[
%B_M^\infty:= \pi_M^*(M^\infty\times \PP^1+\ol{M}\times \{\infty\}),~\ol{U}_M:= \ol{M}\times\PP^1 ~ \backslash ~{\ol{Z}\times \PP^1},
%\]
%\[
%W_M:\text{the strict transform of }\ol{Z}\times\PP^1 \text{ w.r.t. }\ol{B}_M \to \ol{M}\times\PP^1,~B_{M,1}^{(\ol{Z})}:=(\ol{B}_M,B_M^\infty + W_M),
%\]
%\[
%U_M:=(\ol{U}_M,\ol{U}_M \cap (M^\infty\times \PP^1 + \ol{M}\times\{\infty\})),
%\]
%\[
%\ol{E}_M^{(\ol{Z})} :\text{the exceptional divisor of }\pi_M,E_{M,1}^{(\ol{Z})}:=\bigl(\ol{E}_M,(\ol{E}_M\cap B_M^\infty) + (W_M\cap\ol{E}_M)_\red\bigr).
%\]
\end{notation}

The goal of this section is to prove the following theorem.

\begin{thm} \label{Blowupup}
In the notation of Situation~\ref{situ}, %
%Consider modulus pairs $\MpM, \MpZ \in \ulMCorls$ where $\overline{Z}$ is an integral effective Cartier divisor on $\overline{M}$ not contained in $M^{\infty}$, where $Z^{\infty} = M^{\infty}._{\overline{M}}\overline{Z}$, and we suppose $|M^{\infty} +\ol{Z}|$ is a strict normal crossings divisor on $\ol{M}$. Then
there exist a distinguished triangle in $\ulMDM^{\eff}$.
\[
\ulM(E_{M,cl}^{(\ol{Z})}) \to \ulM(Z)\oplus \ulM(B_{M,cl}^{(\ol{Z})}) \to \ulM(M\otimes\bcube) \xrightarrow{+}.
\]
\end{thm}

%Recall that $\MDM^{\eff}$ is a version of $\ulMDM^{\eff}$ constructed using modulus pairs whose total space is proper, cf.\cite[Thm.1]{KM18}, %$\MDM^{\eff}$ is the localising subcategory of $\ulMDM^{\eff}$ generated by motives of modulus pairs whose total space is proper, and that 

%%%%%%%%%%%%%%%%%%%%% Bloch $$$$$$$$$$$$$$$$$$$$$$$$$$$$

\subsection{Special case}

Let $H_0, H_1, H_2$ be the hyperplanes on $\PP^2$ given by $\{[0:\star:\star]\}$, $\{[\star:0:\star]\}$, $\{[\star:\star:0]\}$. We set $b:B\to \PP^2 $ to be the blow up of $\PP^2$ along $H_0\cap H_1$, and set 
$\widetilde{H_0}, \widetilde{H_1}, \widetilde{H_2}$ 
to be the strict transforms of 
$H_0, H_1, H_2$. We set 
\[ B_{cl}:=(B,\widetilde{H_0}+\widetilde{H_2}) \qquad 
%\textrm{ and }
\textrm{ and }\qquad   E_{cl}:=(E,E\cap \widetilde{H_0}), \qquad \]
%We set 
where $E$ is the exceptional divisor of the blow up.

\begin{prop}\label{4.3}
There is a split distinguish triangle
\begin{equation}
\ulM(E_{cl}) \xrightarrow{
\begin{bmatrix}
p & i\\
\end{bmatrix}}
\ulM(\Spec k) \oplus \ulM(B_{cl}) \xrightarrow{
\begin{bmatrix}
j  \\
-b
\end{bmatrix}}\ulM(\PP^2,H_2) \underset{0}{\xrightarrow{+}}
\end{equation}
in $\ulMDM^\eff$, where $i,j$ are the natural closed immersions and $p$ is the natural projection $E \to \Spec k$.
\end{prop}

\begin{proof}
Since $B_{cl}$ has a projection to $E_{cl}$ which is a cube bundle, $i$ is an isomorphism. Additionally $j$ is also an isomorphism since $(\PP^2,H_2)$ is contractible \cite[Lemma 10]{KS19}. 
\end{proof}

\subsection{Proof of Theorem \ref{Blowupup}}

\begin{thm}\label{AAAAA}
There is a distinguish triangle. %
\[
\ulM(E_{(\AA^1,\emptyset),cl}^{(\{0\})}) \to \ulM(\{0\})\oplus \ulM(B_{(\AA^1,\emptyset),cl}^{(\{0\})}) \to \ulM((\AA^1,\emptyset)\otimes\bcube) \xrightarrow{+}.
\]
\end{thm}

A log version of the argument below appeared independently in Binda-Park-{\O}stv{\ae}r (see \cite[Proposition~7.2.5]{BPO}).

\begin{proof}
We set $T$ to be the blow up of $\PP^2$ at $H_0\cap H_2$, let $f$ be the exceptional divisor, and $h_i$ be the strict transforms of the $H_i$. 
We set $T'$ to be the blow up of $T$ at $h_0 \cap h_1$, let $e$ be the exceptional divisor, and the $\widetilde{h_i}$ be strict transforms of the $h_i$ and $\widetilde{f}$ be the strict transform of $f$. 
In particular, $T'$ is same as the blow up of $B$ at $\widetilde{H_0} \cap \widetilde{H_2}$. The fans of these toric varieties are as follows:
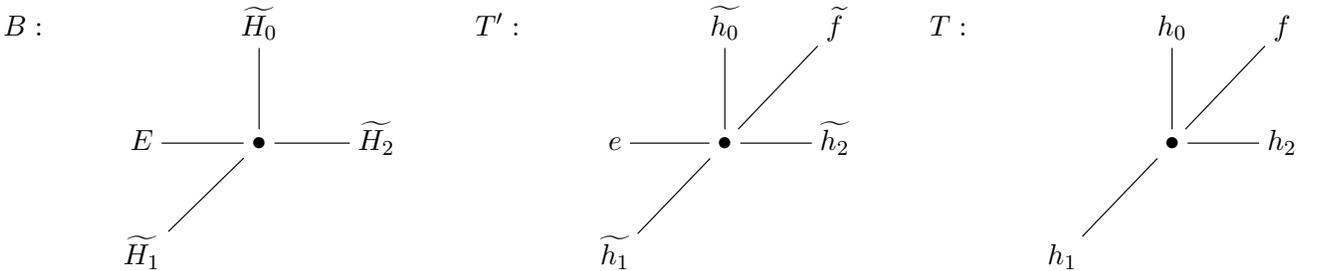
\[ \xymatrix{
B:& & \widetilde{H_0} & &
T':& & \widetilde{h_0} & \widetilde{f} &
T:& & h_0 & f
\\
& E \ar@{-}[r] &\ar@{-}[r] \bullet \ar@{-}[u]   \ar@{-}[dl] 
& \widetilde{H_2} &
& e \ar@{-}[r]&\ar@{-}[r] \bullet \ar@{-}[u]  \ar@{-}[ru] \ar@{-}[dl] 
& \widetilde{h_2} &
& &\ar@{-}[r] \bullet \ar@{-}[u]  \ar@{-}[dl] \ar@{-}[ur]
& h_2 
\\
&\widetilde{H_1} && &
&\widetilde{h_1} && &
&h_1 &&
%\\
%&& B && && T' && && T
}\]

The following triangle is isomorphic to it in Proposition~\ref{4.3}, since this is obtained by blowing up inside the modulus.
\begin{equation}\label{BlclT}
\ulM(E_{cl}) \to
\ulM(\Spec k) \oplus \ulM(T', \widetilde{h_2} + \widetilde{h_0} + \widetilde{f}) \to \ulM(T,h_2+f) \underset{0}{\xrightarrow{+}}
\end{equation}

Notice that there is a canonical isomorphism of toric surfaces $T \setminus h_2 \cong \A^1 \times \P^1$ inducing an isomorphism of modulus pairs $(T \setminus h_2, f \setminus (f \cap h_2)) \cong (\A^1,\emptyset) \otimes \bcube$. Furthermore, pulling back the square that give rise to \eqref{BlclT} along $\A^1 \times \P^1 \to T$ produces the triangle in the statement.

Since $\{T \setminus h_2 \to T, T \setminus h_0 \to T\}$ is a Zariski covering, by Mayer-Vietoris, to show that the triangle in the statement is distinguished, it suffices to show that the triangle associated to $T, T \setminus h_0$, and $T \setminus (h_0 \cup h_2)$ is distinguished, cf. Lem.~\ref{nine}. We have just seen that the triangle associated to $T$ is isomorphic distinguish triangle (\ref{BlclT}). On the other hand, since the centre of the blowup is contained in $h_0$, the triangle coming from $T \setminus h_0$ and $T \setminus (h_0 \cup h_2)$ is trivially distinguished.
%
%{\color{red} ***add more details ***} 
\end{proof}

\begin{thm}\label{base}
For any modulus pair $(\ol{Y}, Y^\infty) \in \ulMCor$ such that $\ol{Y}$ is smooth and $Y^\infty$ is a strict normal crossings divisor, there is a distinguish triangle: %
\[
\ulM(E_{Y\otimes(\AA^1,\emptyset),*}^{(Y\otimes\{0\})}) \to \ulM(Y\otimes\{0\})\oplus \ulM(B_{Y\otimes(\AA^1,\emptyset),*}^{(Y\otimes\{0\})}) \to \ulM(Y\otimes(\AA^1,\emptyset)\otimes\bcube) \xrightarrow{+}.
\]
\end{thm}
\begin{proof}
Since 
\begin{eqnarray*}
B_{Y\otimes(\AA^1,\emptyset),cl}^{(Y\otimes\{0\})}= Y\otimes B_{(\AA^1,\emptyset),cl}^{(\{0\})},
\end{eqnarray*}
the triangle in the statement is the triangle from Theorem~\ref{AAAAA} tensored by $(\ol{Y}, Y^\infty)$.
\end{proof}

\begin{situ} \label{situ2}
Let $f:(\ol{N},N^\infty) \to (\ol{M},M^\infty)$ be an \'etale morphism (i.e., $f$ is induced by an \'etale morphism $\ol{f}:\ol{N} \to \ol{M}$ and is minimal) such that $f$ induces an isomorphism $f^{-1}\ol{Z} \to \ol{Z}$.
\end{situ}

\begin{notation}\label{otimesnote}
In Situation \ref{situ2}, we pullback everything from Notation~\ref{otimessitu} along $f$. That is, we set 
\[ N, f^{-1}Z, \ol{B}_N^{(f^{-1}\ol{Z})}, \pi_N, B_N^\infty, W_N, B_{N,cl}^{(f^{-1}\ol{Z})}, \ol{U}_N, U_N, \ol{E}_N^{(f^{-1}\ol{Z})}, E_{N,cl}^{(f^{-1}\ol{Z})} \]
to be the pullbacks of 
\[ M, Z, \ol{B}_M^{(\ol{Z})}, \pi_M, B_M^\infty, W_M, B_{M,op}^{(\ol{Z})}, B_{M,cl}^{(\ol{Z})}, \ol{U}_M, U_M, \ol{E}_M^{(\ol{Z})}, E_{M,op}^{(, \ol{Z})}, E_{M,cl}^{(\ol{Z})} \]
along $\ol{f}: \ol{N} \to \ol{M}$. Explicitly, 
\begin{align*}
N &:= (\ol{N}, N^\infty), \\
f^{-1}Z &:= (f^{-1}\ol{Z},f^{-1}Z\cdot_{\ol{N}} N^\infty) \\
\ol{B}_N^{(f^{-1}\ol{Z})}\xrightarrow{\pi_N}\ol{N}\times\PP^1 &:\text{the blow up of } \ol{N}\times\PP^1 \text{ along } f^{-1}\ol{Z}\times \{0\}, \\
B_N^\infty &:= \pi_N^*(N^\infty\times \PP^1 + \ol{N}\times\{\infty\}) \\
W_N &:\text{the strict transform of }f^{-1}\ol{Z}\times\PP^1 \text{ w.r.t. }\ol{B}_N \to \ol{N}\times\PP^1, \\
B_{N,cl}^{(f^{-1}\ol{Z})}&:=(\ol{B}_N,B_N^\infty + W_N), \\
\ol{U}_N &:= \ol{N}\times\PP^1 ~ \backslash ~{f^{-1}\ol{Z}\times \PP^1}, \\
U_N &:=(\ol{U}_N,\ol{U}_N \cap (N^\infty\times \PP^1 + \ol{N}\times\{\infty\})) \\
\ol{E}_N^{(f^{-1}\ol{Z})} &:\text{the exceptional divisor of }\pi_N, \\
E_{N,cl}^{(f^{-1}\ol{Z})} &:= (\ol{E}_N,\ol{E}_N\cap B_N^\infty + (W_N\cap\ol{E}_N)_).
\end{align*}

%\[
%N:=(\ol{N},N^\infty),~f^{-1}Z:=(f^{-1}\ol{Z},f^{-1}Z\cdot_{\ol{N}} N^\infty)
%\]
%%\[
%%\ol{B}_M\xrightarrow{\pi_M}\ol{M}\times\PP^1 :\text{the blow up of } \ol{M}\times\PP^1 \text{ with along }\ol{Z}\times \{0\},
%%\]
%\[
%\ol{B}_N^{(f^{-1}\ol{Z})}\xrightarrow{\pi_N}\ol{N}\times\PP^1 :\text{the blow up of } \ol{N}\times\PP^1 \text{ with along } f^{-1}\ol{Z}\times \{0\},
%\]
%\[
%%B_M^\infty:= \pi_M^*(M^\infty\times \PP^1+\ol{M}\times \{\infty\}),~
%B_N^\infty:= \pi_N^*(N^\infty\times \PP^1 + \ol{N}\times\{\infty\})\]\[
%W_N:\text{the strict transform of }f^{-1}\ol{Z}\times\PP^1 \text{ w.r.t. }\ol{B}_N \to \ol{N}\times\PP^1,
%\]
%\[
%%B_{M,n}:=(\ol{B}_M,B_M^\infty + nW_M),~
%B_{N,1}^{(f^{-1}\ol{Z})}:=(\ol{B}_N,B_N^\infty + W_N),~
%%\ol{U}_M:= \ol{M}\times\PP^1 ~ \backslash ~{\ol{Z}\times \PP^1} ,~
%\ol{U}_N:= \ol{N}\times\PP^1 ~ \backslash ~{f^{-1}\ol{Z}\times \PP^1},~%U_M:=(\ol{U}_M,\ol{U}_M \cap (M^\infty\times \PP^1 + \ol{M}\times\{\infty\})),~
%\]\[U_N:=(\ol{U}_N,\ol{U}_N \cap (N^\infty\times \PP^1 + \ol{N}\times\{\infty\}))
%\]
%\[
%\ol{E}_N^{(f^{-1}\ol{Z})} :\text{the exceptional divisor of }\pi_N,~%E_{M,n}:=(\ol{E}_M,\ol{E}_M\cap B_M^\infty + n(W_M\cap\ol{E}_M)_\red),~
%E_{N,1}^{(f^{-1}\ol{Z})}:=(\ol{E}_N,\ol{E}_N\cap B_N^\infty + (W_N\cap\ol{E}_N)_\red).
%\]
\end{notation}

\begin{prop}
In Situation \ref{situ2} and Notation \ref{otimesnote},
\[
\xymatrix{
U_N \ar[r] \ar[d] & U_M \ar[d]\\
B_{N,cl}^{(f^{-1}\ol{Z})} \ar[r] & B_{M,cl}^{(\ol{Z})}\\
}
\]
is elementary Nisnevich square.
\end{prop}

\begin{proof}
All morphisms in the square are minimal. By definitions,  $(B_{M,cl}^{(\ol{Z})} ~ \backslash ~ U_M) = (W_M \cup \ol{E}_M)$,  $(B_{N,cl}^{(f^{-1}\ol{Z})} ~ \backslash ~ U_N) = (W_N \cup \ol{E}_N)$.
\end{proof}

\begin{cor}\label{NMdeform}
In Situation \ref{situ2} and Notation \ref{otimesnote}, the image under the functor $\ulM$ of the square
\[
\xymatrix{
B_{N,cl}^{(f^{-1}\ol{Z})} \ar[r] \ar[d] & B_{M,cl}^{(\ol{Z})}\ar[d]\\
N\otimes \bcube \ar[r]  & M\otimes\bcube \\
}
\]
is a homotopy Cartesian in $\ulMDM^\eff$.
\end{cor}

\begin{notation}
In Notation \ref{otimessitu} Consider the following ``weighted'' blowup formulas.
%\footnotesize
\[
(WBU)_{Z\to M}^{cl}: \text{ the object } \ulM \left (E_{M,cl}^{(\ol{Z})} {\to} Z \oplus B_{M,cl}^{(\ol{Z})} {\to} M\otimes\bcube \right ) \text{ is isomorphic to zero in }\ulMDM^\eff.
\]
\normalsize
\end{notation}
\begin{prop}\label{etaleetale}
In Situation \ref{situ2},  $(WBU)_{Z\to M}^*$ is true if and only if $(WBU)_{f^{-1}Z\to N}^*$ is true.
\end{prop}

\begin{proof}
The following diagram commutes in $\ulMCor$.
\[
\xymatrix{
\Ztr(E_{M,cl}^{(\ol{Z})})(\simeq \Ztr(E_{N,cl}^{(f^{-1}\ol{Z})})) \ar[r] \ar[d] & \Ztr(Z)(\simeq \Ztr(f^{-1}Z)) \ar[d]\\
\Ztr(B_{N,cl}^{(f^{-1}\ol{Z})}) \ar[r]\ar[d] &\Ztr(N\otimes\bcube) \ar[d]\\
\Ztr(B_{M,cl}^{(\ol{Z})}) \ar[r] & \Ztr(M\otimes\bcube)\\
}
\]
By Corollary~\ref{NMdeform} we know the lower square is a homotopy Cartesian in $\ulMDM^\eff$, so the outer square is a homotopy Cartesian iff the upper square is.
\end{proof}

The following lemma is proved in \cite{KS19}.

\begin{lemma}[{\cite[Lemma 8]{KS19}}] \label{decompose}
In Situation~\ref{situ}, up to replacing $\OL{M}, \OL{Z}, M^{\infty}$ by $\OL{V}, \OL{V} \cap \OL{Z}, \OL{V} \cap M^{\infty}$ for some open neighborhood $x \in\OL{V}$, there exists an \'etale morphism $q:\OL{M} \to \mathbb{A}^{m}$ such that $\OL{Z} = q^{-1}(\mathbb{A}^{m-1}\times\{0\})$ and $M^{\infty} = q^{-1}(\{T_1^{d_1} ...T_s^{d_s} = 0\})$ where $T_{i}$ are the coordinates of $\mathbb{A}^{m}$.
\end{lemma}

Now we have enough pieces to prove Theorem~\ref{Blowupup}.

\begin{proof}[Proof of Theorem \ref{Blowupup}]
For any open covering $\{\ol{U}_i \to \ol{M}\}_i$, by the Mayer-Vietoris sequence we obtain that $(WBU)_{Z\to M}^*$ is true if $(WBU)_{Z\cap\ol{U}\to (\ol{U},\ol{U}\cap M^\infty)}^*$ is true for all open sub schemes $\ol{U} \subset \ol{U}_i$ for all $i$. By Proposition \ref{etaleetale} and Lemma~\ref{decompose} we can reduce the claim to the case $(WBU)_{Z\otimes\{0\} \to Z\otimes (\AA^1,\emptyset)}^*$, but this was proved in Theorem \ref{base}.
\end{proof}

%\end{figu}
%%%%%%%%%%%%%%%%%%%%% chapter 1 %%%%%%%%%%%%%%%%%%%%%%%%
\section{Construction of the Tame Gysin map} \label{sec:tameGysinMapConstruction}

\subsection{Notation}

We continue with $\fM=(\ol{M},\Minf)$ and $\fZ=(\ol{Z},Z^{\infty})$
%,\text{ } \fM_Z=(\ol{M},M^\infty+\ol{Z})
%\]
satisfying the hypotheses of Situation~\ref{situ}. We furthermore drop all the indexes "$M$" from the notation of Notation~\ref{otimessitu}. So 
%\[
%=\Cone(\ulM(\mathbb{P}(N_{\ol{Z}}\ol{M} \oplus \mathcal{O}), \pi^* Z^\infty + \{\infty\}_{\ol{Z}}) \to \ulM(\mathbb{P}(N_{\ol{Z}}\ol{M} \oplus \mathcal{O}), \pi^* Z^\infty) )
%\]
\[
\ol{B}:\text{is the blow-up of } \ol{M} \times \PP^1 \text{ along } \ol{Z} \times \{0\}, \textrm{ and }
\]
\[
E:\text{is the exceptional divisor of }q,
\]
so we have a Cartesian square
\[ \xymatrix{
E \ar[r] \ar[d]_\pi & \ol{B} \ar[d]^q \\
\ol{Z} \times \{0\} \ar[r] & \ol{M} \times \PP^1.
}
\]
We put  
\[ \mf{B}:= (\ol{B},q^*(M^\infty\times\PP^1+\ol{M}\times \{\infty\})),
\]
\[
\mf{B}_{Z,cl}:= (\ol{B},q^*(M^\infty\times\PP^1+\ol{M}\times \{\infty\})+\widetilde{(\ol{Z}\times\PP^1))} \]
\[
\mf{E}_{Z,cl}:=(\PP(N_{Z}M\oplus \mathcal{O}),\pi^*Z^\infty + \PP(0 \oplus \mathcal{O})).
\]

\begin{thm}\label{KS}
There is a distinguished triangle in $\ulMDM^\eff$
\[
\ulM(\mf{E}_{Z,cl}) \to \ulM(B_{Z,cl})\oplus \ulM(\mf{Z}) \to  \ulM(\mf{M}\otimes \bcube) \xrightarrow{+}
\]
and isomorphism
\[
    Th(N_ZM,cl) = \Cone(\ulM(\mf{E}_{Z,cl}) \to \ulM(\mf{E
    }_{})) \simeq  \Cone(\ulM(\mf{B}_{Z,cl}) \to \ulM(\mf{B}_{})). 
\]
\end{thm}
\begin{proof}
The first claim is Theorem \ref{Blowupup}, the second claim  follows from the first and the blow up formula 
\[
\ulM(\mf{E}) \to \ulM(\mf{B } )\oplus \ulM(\mf{Z}) \to  \ulM(\mf{M}\otimes \bcube) \xrightarrow{+}
\]
proved in \cite[Theorem, page 1]{KS19}.
\end{proof}

\subsection{Geometrical study}
We write $i_1$ for the natural embedding of schemes $\ol{M}\times\{1\}$ to $\ol{B}$. The embedding $i_1$ defines $i:\mf{M}_{} \to \mf{B}_{}$ and $\tilde{i}:\mf{M}_{Z,cl} \to \mf{B}_{Z,cl}$ in $\ulMCor$.

\[ \xymatrix{
\mf{M}_{Z,cl} \ar[r]^{\tilde{i}} \ar[d]  
& \mf{B}_{Z,cl} \ar[d] \\
\mf{M}_{} \ar[r]^i
& \mf{B}\\
} \]

The diagram gives us the morphism $\Ztr(\mf{M}\slash \mf{M}_{Z,cl}) \to \Ztr(\mf{B}\slash \mf{B}_{Z,cl})$ in $PSh(\ulMCor)$ where we write
\[ \Ztr(\mf{M}\slash \mf{M}_{Z,cl}) := coker\biggl ( \Ztr(\mf{M}_{Z,cl}) \to \Ztr(\mf{M}) \biggr ) \]
etc., in $PSh(\ulMCor)$. Note that since $\mf{M}_{Z,cl} \to \mf{M}$ are monomorphisms, the image of $\Ztr(\mf{M}\slash \mf{M}_{Z,cl})$ in $\ulMDM^\eff$ is the cone of the image of $\Ztr(\mf{M}_{Z,cl}) \to \Ztr(\mf{M})$. Composing with the isomorphism 
\[ \Ztr(\mf{B}\slash \mf{B}_{Z,cl}) \stackrel{\sim}{\leftarrow} \Ztr(\mf{E}\slash \mf{E}_{Z,cl}) = Th(N_ZM,cl) \]
from Theorem \ref{KS}, one gets a morphism:
\begin{equation*}
\beta(\mf{M}\slash Z,cl) : \ulM(\mf{M}_{}\slash \mf{M}_{Z,cl}) \to \Ztr(\mf{B}\slash \mf{B}_{Z,cl}) \to Th(N_ZM,cl).
\end{equation*}
We call this morphism the \emph{closed Gysin map} associated with $\mf{M}$ and $Z$.
\begin{lemma}\label{compati} We have the following.
\begin{enumerate} 
    \item[$(0)$] $\beta\bigl((\mathbb{A}^{1},\emptyset)\slash \{0\},cl\bigr): \ulM((\mathbb{A}^{1},\emptyset)\slash (\AA^1,\{0\})) \to Th(N_{\{0\}}(\AA^1,\emptyset),cl)$ is an isomorphism.

    \item[$(1)$] For any \'etale morphism $e:\mf{M}'=(\ol{M}',e^{*}M^{\infty}) \to \mf{M}$, set $\mf{Z}'=(e^{-1}\ol{Z},e^{*}Z^{\infty})$. Then the diagram 
\[
\xymatrix{
\ulM(\mf{M
}'_{}/{\mf{M}'_{Z',cl}}) \ar[rr]^{\beta(\mf{M}'\slash \mf{Z}',cl)} \ar[d]& & Th(N_{Z'}M',cl) \ar[d] \\
\ulM(\mf{M
}_{}/{\mf{M}_{Z,cl}}) \ar[rr]_{\beta(\mf{M}\slash \mf{Z},cl)} & &Th(N_{Z}M,cl)
}
\] 
commutes.

    \item[$(2)$] For any modulus pair $\mathfrak{Y}=(\ol{Y},Y^{\infty})$ with $\ol{Y}$ smooth and $Y^\infty$ a strict normal crossings divisor, we have
\[
\beta(\mf{M}\otimes \mf{Y}\slash \mf{Z} \otimes \mf{Y},cl) = \beta(\mf{M}\slash \mf{Z},cl) \otimes Id_{\underline{M}(\mf{Y})}.
\]
\end{enumerate}
\end{lemma}

\begin{proof}
Part 1. We take
\[
\OL{B}':\text{blow-up of } \OL{M}' \times \mathbb{P}^{1} \text{ with along } e^{-1}\OL{Z} \times \{0\},
\]
and
\[
\mf{B}'=\bigl(\OL{B}',{q'}^{*}(e^{*}M^{\infty}\times \mathbb{P}^{1}) + {q'}^{*}(\OL{M}'\times \{\infty\})\bigr),
\]
\[
\mf{B}'_{\mf{Z}',cl}=\bigl(\OL{B}',{q'}^{*}(e^{*}M^{\infty}\times \mathbb{P}^{1}) + {q'}^{*}(\OL{M}'\times \{\infty\})+\widetilde{(\ol{Z}'\times \PP^1)}\bigr)
\]
Since the morphism $e$ is \'etale, $e^{-1}\OL{Z}$ is also smooth. Now there is a natural map $\mf{B}' \to \mf{B}$, and we have the following commutative diagram in $\ulMCor$.
%\scriptsize
\[\xymatrix@!=6pt{
& \mf{M}'_{\mf{Z}'} \ar[rr] \ar[dd]|\hole \ar[dl] & & \mf{M}'_{}\ar[dd]^{} \ar[dl] \\
\mf{M}_{\mf{Z}} \ar[dd] \ar[rr] &  & \mf{M}_{} \ar[dd]^(0.3){} & \\
& \mf{B}'_{\mf{Z}'} \ar[rr]|-\hole \ar[ld] & & \mf{B}'_{} \ar[ld] \\
 \mf{B}_{\mf{Z}} \ar[rr] & & \mf{B}_{}& \\
}
\]\normalsize
The diagram gives us the commutative diagram in $PSh(\ulMCor)$.
\[ \xymatrix{
\Ztr(\mf{M}'_{}\slash \mf{M}'_{\mf{Z}',cl}) \ar[r] \ar[d] & \Ztr(\mf{M}_{}\slash \mf{M}_{\mf{Z},cl}) \ar[d]\\
\Ztr(\mf{B}'_{}\slash \mf{B}'_{\mf{Z}',cl}) \ar[r] & \Ztr(\mf{B}_{}\slash \mf{B}_{\mf{Z},cl})
} \]
The same argument shows that the square
\[ \xymatrix{
\Ztr(\mf{E}'_{}\slash \mf{E}'_{\mf{Z}',cl}) \ar[r] \ar[d] & \Ztr(\mf{E}_{}\slash \mf{E}_{\mf{Z},cl}) \ar[d]\\
\Ztr(\mf{B}'_{}\slash \mf{B}'_{\mf{Z}',cl}) \ar[r] & \Ztr(\mf{B}_{}\slash \mf{B}_{\mf{Z}.cl})
} \]
is commutative.

Part 2. The blow-up of $\OL{M}\times\OL{Y}\times\mathbb{P}^{1}$  along $\OL{Z}\times\OL{Y}\times\{0\}$ is isomorphic to $\OL{B} \times \OL{Y}$, so the proof is completed.

Part 0. Set $\eta_{op}$ resp. $\eta_{cl}$ to be the composition of the $1$-section 
$\AA^1 \times \{1\} \hookrightarrow  \ol{B}_{\{0\}}$ 
resp. 
$(\AA^1 \setminus \{0\}) \times \{1\} \hookrightarrow  \ol{B}_{\{0\}}$, 
and the retraction $\ol{B}_{\{0\}}\to \ol{E}_{\{0\}}$ (see the diagrams below on the left). Note that these are open immersions. Let $\eta_{0,cl}, \eta_1$ be the induced morphisms on modulus pairs (see the square below on the right).
\[ 
\xymatrix{
& E_{\{0\},cl}  \ar[d] \ar[rd]&\\
\AA^1\ar[r]^{i}\ar[rd]& B_{\{0\},cl} \ar[r]  \ar[d]&   E_{\{0\},cl} 
&
\ulM(\AA^1,\{0\}) \ar[r] \ar[d]^{\eta_{1.cl}}& \ulM(\AA^1,\emptyset)\ar[d]^{\eta_{0}}
\\
&(\AA^1,\emptyset)\otimes \bcube  &
&
\ulM(\mf{E}_{\{0\},cl})\ar[r]& \ulM(\mf{E}_{})
} \]
%\begin{floatingfigure}[r]{5cm}
%\setlength{\unitlength}{0.8cm}
%\centering
%\xymatrix{
%\ulM(\AA^1,\{0\}) \ar[r] \ar[d]^{\eta_1}& \ulM(\AA^1,\emptyset)\ar[d]^{\eta_0}\\
%\ulM(\mf{E}_{\{0\}})\ar[r]&
%\ulM(\mf{E}_{})}
%%\caption{図の説明}\label{fig:01}
%\end{floatingfigure}%%\
Since both of right side squares satisfy the condition of Proposition \ref{Fetexcion}, these squares are homotopy Cartesian in $\ulMDM^\eff$.    
\end{proof}

%%%%%%%%%%%%%%%%%%%%%%%%%%%% chapter 4 %%%%%%%%%%%%%%%%%%%%%%%%%%%%%%%%%%%%%%%%%%%%%%%%%%%%%
\section{Proof of main theorems}

In this section, we use the notation of Section 5.1, and we prove that the Gysin maps defined in Section 5.2
\[
\beta(\fM\slash \fZ,cl):
\ulM(\fM_{}/\fM_{\fZ,cl}) \to 
Th(N_ZM,cl),
\]
are isomorphisms. %It follows that the open Gysin triangle is a distinguished triangle.

\begin{lemma}\label{locloc}
The Gysin maps $\beta(\fM\slash \fZ,*)$ is an isomorphism if there is an open Zariski cover $\{\overline{V}_i \to \overline{M}\}_{i = 1}^l$ such that for all $i$, the Gysin maps $\beta(%
(\overline{V},\overline{V} \cap M^\infty)\slash %
(\overline{V} \cap \overline{Z}, \overline{V} \cap M^\infty \cap \overline{Z}))$ %
associated to the intersections $\overline{V} = \cap_{j \in J} \overline{V}_j$ are isomorphisms for all nonempty $J \subseteq I$.
\end{lemma}

\begin{proof}
By induction on $l$ it suffices to consider the $l = 2$ case. We take an open covering $\OL{V}_1\cup\OL{V}_2 = \OL{M}$. Now we set 
\[
\xymatrix{
\mf{V}_i=(\OL{V}_i,\OL{V}_i\cap M^{\infty}), & {\mf{V}_{\fZ,i,cl}}=(\OL{V}_i, \OL{V}_i \cap M^{\infty} + \OL{V}_i \cap \ol{Z}),
}
\]
\[
\xymatrix{
\OL{V}_{12}=\OL{V}_1 \cap \OL{V}_2, & \mf{V}_{12}=(\OL{V}_{12},\OL{V}_{12}\cap M^{\infty}), 
}
\] 
\[
{\mf{V}_{\fZ,12,cl}} =(\OL{V}_{12}, \OL{V}_{12} \cap M^{\infty} + \OL{V}_{12} \cap \ol{Z}). 
\]

We have the following diagram in $PSh(\ulMCor)$,\scriptsize
\[\xymatrix{
&0 \ar[d] &0 \ar[d] &0 \ar[d] &\\ 
 & 
\Zt({\mf{V}_{\fZ,12,*}}) \ar[r] \ar[d]^{i_{12}^*} &  
\Zt({\mf{V}_{\fZ,1,*}}) \oplus 
\Zt({\mf{V}_{\fZ,2,*}}) \ar[d]^{i_1^{*} \oplus i_2^{*}}\ar[r] &  
\Zt(\fM_{\fZ,*}) \ar[d]^{i_{\fM}}  & \\
 & 
\Zt(\mf{V}_{12})\ar[r] \ar[d] &  
\Zt(\mf{V}_{1}) \oplus 
\Zt(\mf{V}_{2}) \ar[d] \ar[r] &  
\Zt(\fM_{}) \ar[d]  & \\
 &  \text{Coker}(i_{12}^*) \ar[r] \ar[d] &   \text{Coker}({i_{1}^*}) \oplus \text{Coker}({i_{2}^*}) \ar[r] \ar[d]&  \text{Coker}(i_{\fM}^*)  \ar[d] &  \\
&0&0&0&
}
\]\normalsize
where the compositions of all columns and the two top rows are zero, and the bottom row maps are uniquely determined by the middle row maps.

By Lemma \ref{nine} we get the following distinguish triangle in $\ulMDM^\eff$
\[
\ulM(\mf{V}_{\fZ,12,*}/{\mf{V}_{\fZ,12,*}}) \to 
\ulM(\mf{V}_{\fZ,1,*}/{\mf{V}_{\fZ,1,*}}) \oplus 
\ulM(\mf{V}_{\fZ,2,*}/{\mf{V}_{\fZ,2,*}}) \to 
\ulM(\fM_{\fZ}/{\fM_{\fZ}}) \xrightarrow{+}.
\]

Same argument can be applied for Thom space, so we obtain the following distinguish triangle in $\ulMDM^\eff$ 
\[
Th(N_{Z_{12}}\mf{V}_{12},*) \to 
Th(N_{Z_{1}}\mf{V}_{1},*) \oplus 
Th(N_{Z_{2}}\mf{V}_{2},*) \to 
Th(N_{Z}\mf{M},*) \xrightarrow{+}.
\]
%\[
%Th(N_{Z_{12}}\mf{V}_{12},n,m) \to 
%Th(N_{Z_{1}}\mf{V}_{1},n,m) \oplus 
%Th(N_{Z_{2}}\mf{V}_{2},n,m) \to 
%Th(N_{Z}\mf{M},n,m) \xrightarrow{+}.
%\]

By Lemma~\ref{compati}, we know that Gysin maps are compatible with open immersions, so the proof follows from the triangulated category axioms.  
\end{proof}

\begin{lemma} \label{NOREFCOR}
In the situation Theorem \ref{PEPEPE},  $\beta((\ol{N},N^\infty)\slash (f^{-1}\ol{Z},Z'^\infty),cl)$ is an isomorphism if and only if $\beta((\ol{M},M^\infty)\slash (\ol{Z},Z^\infty),cl) $ is isomorphism.
\end{lemma}

\begin{proof}
by Lemma 3.2. (1), we have the following commutative diagram
\[
\xymatrix{
\ulM(\mf{N
}_{}/{\mf{N}_{f^{-1}\fZ,cl}}) \ar[r] \ar[d] & Th(N_{f^{-1}Z}N,cl) \ar[d] \\
\ulM(\mf{M
}_{}/{\fM_{\fZ,cl}}) \ar[r] & Th(N_ZM,cl)
}
\]
where the vertical maps are isomorphisms by Theorem \ref{PEPEPE} and Corollary \ref{PEPP}. So if one of the horizon maps is an isomorphism then the other is also an isomorphism.
\end{proof}

%\begin{remark} By \cite[Proposition 1.11.1]{KSY15}, the modulus pair $\deAm$ is equal to the fibre product
%\[\Bigl( \deAl \Bigr) \times_{(\Spec{k},\emptyset)} (\Spec{\mbA},\emptyset),\]
%since $(\mbA,\emptyset)\to(\Spec k,\emptyset)$ is minimal and smooth.
%\end{remark} 

Now we have a proof of the main theorem.

\begin{proof}[Proof that the Gysin triangles are distinguished in $\ulMDM^\eff$.] It suffices to show the Gysin morphisms are isomorphisms. By Lemma~\ref{locloc}, Lemma~\ref{NOREFCOR}, and Lemma~\ref{decompose}, we can assume that there is an {\'e}tale map $\OL{f}: \OL{M} \to \mathbb{A}^m$ such that $M^\infty = \OL{f}^*E$ and $\OL{Z} = \OL{f}^*(\mathbb{A}^{m-1} \times \{0\})$ where 
%\[ E=%
%d_{1}\{0\} \times \mathbb{A}^{m-1} + %
%\mbA \times \{0\} \times \mathbb{A}^{m-2} + %
%\cdots + %
%\mathbb{A}^{m-2} \times d_{m-1}\{0\} \times \mathbb{A}^{1}. \] %
%for some $d_1, \dots, d_{m-1} \geq 0$. %
%Then
\[ \deAm = (\mbAm,E), \]
and we write %
$E_{0} = %
E \times_{\mathbb{A}^{m}} (\mathbb{A}^{m-1}\times\{0\})$
%d_{1}\{0\} \times \mathbb{A}^{m-2} \times \{0\} + %
%\mbA \times d_{2}\{0\} \times \mathbb{A}^{m-3} \times \{0\} + %
%\cdots + %
%\mathbb{A}^{m-2} \times d_{m-1}\{0\} \times \{0\}$, %
so
\[ \deAmo=(\mathbb{A}^{m-1} {\times} \{0\},E_{0}). \]

%Now $E_{0} = E \cdot_{\mathbb{A}^{m}} (\mathbb{A}^{m-1}\times\{0\}) =  E \times_{\mathbb{A}^{m}} (\mathbb{A}^{m-1}\times\{0\})$, so the map $(\mathbb{A}^{m-1}\times \{0\}, E_{0}) \hookrightarrow (\mathbb{A}^{m},E)$ is a closed sub modulus pair. 
Now we have a Cartesian cubic diagram\scriptsize
\[\xymatrix{
& Z^{\infty} \ar[rr] \ar[dd]|-\hole \ar[dl] & & \OL{Z} \ar[dd]^{\OL{f}_{Z}} \ar[dl] \\
M^{\infty} \ar[dd] \ar[rr] &  & \OL{M} \ar[dd]^(0.3){\OL{f}} & \\
& E_{0} \ar[rr]|-\hole \ar[ld] & & \mathbb{A}^{m-1}\times \{0\} \ar[ld] \\
E \ar[rr] & & \mbAm & \\
}
\]\normalsize

\newcommand{\Mpp}{(\Spec k, \emptyset)}

By the above diagram, we know $f_{Z}:\MpZ \to (\mathbb{A}^{m-1}\times \{0\},E_{0})$ is a minimal \'etale map. Now we consider the fibre product,

\[
\ol{X} ~:=~  \ol{M} \times_{\mathbb{A}^m} (\ol{Z} \times_{\Spec k} \mathbb{A}^1 )
\]
and 
\[
X^\infty ~:=~ \ol{\pi}^*M^\infty,
\]
where $\ol{\pi}$ is a canonical morphism $\ol{X} \to \ol{M}$. Now by \cite[Theorem 4.10]{SV00}, we have a  diagram $(\Omega)$ in $\Sm(k)$,\scriptsize
\[\xymatrix{
\OL{Z} \ \ar@{}[drr]|\square \ar[rr] \ar[d]_{\Delta_{\OL{Z}/\mathbb{A}^{m-1}}} & & \OL{X}' \ar[rrd]^{p_{2}} \ar[d]_{i} \ar@/^18pt/[dd]^(0.6){p_1}|\hole& & (\Omega) \\
\OL{Z}\times_{\mathbb{A}^{m-1}} \OL{Z} \ar[rr] \ar@{}[drr]|\square \ar[d] & & \OL{X} \ar[d]_{\OL{\pi}} \ar[rr]^{p}  \ar@{}[drr]|\square & & \OL{Z}\times \mbA \ar[d]^{\OL{f}_Z \times id_{\mbA}} \\
\OL{Z} \ar[rr] & & \OL{M} \ar[rr]^{} \ar[rr] & &  \mbAm  \\
}
\]\normalsize
where $i:\OL{X}' \to \OL{X}$ is an open immersion, and $p_{2}^{-1}(\OL{Z}\times \{0\})= \OL{Z}$. By Lemma~\ref{NOREFCOR}, $\beta(\mf{M}\slash\mf{M}_{\mf{Z}},cl)$ is an isomorphism if and only if $\beta_{cl}: (\ulM\Bigl(\MpZ \otimes (\mbA,\emptyset) \Bigr) / \ulM\Bigl(\MpZ \otimes (\mbA,\{0\})\Bigr)) \to Th(N_{Z}\AA^1_Z,cl)$ is an isomorphism. By Lemma \ref{compati} (2) this $\beta_{cl}$ is the image of $ \ulM((\AA^1, \emptyset) / (\AA^1, \{0\})) \to Th(N_{\{0\}}\AA^1,cl)$ under $(\ol{Z}, Z^\infty) \otimes -$. But this is an isomorphism by Lemma \ref{compati}(0). %
%
%Finally we show the tame Gysin triangles are distinguished in $\ulMDMNis$. Since we know the exact functor $\ulMDM^\eff \to \ulMDMNis$, by Tame Gysin triangle in $\ulMDM^\eff$, we obtain the claim.
%
\end{proof}

%Set $\OL{U}' = p_1^{-1}\OL{U}$. By Theorem \ref{PEPEPE} we get 
%\begin{eqnarray*} 
%{\text{Cone}}(\underline{M}(\OL{M},M^{\infty}+n\ol{Z}) \to \underline{M}\MpM ) \simeq \qquad\qquad\qquad\qquad\qquad\\
%{\text{Cone}}(\underline{M}(\OL{X}',p_1^{*}M^{\infty} + np_1^*\ol{Z}) \to \underline{M}(\OL{X}',p_1^{*}M^{\infty}))\simeq  \qquad\qquad\qquad\qquad  \\
%{\text{Cone}}\biggl(
%\underline{M}\Bigl(\MpZ \otimes (\mbA,n\{0\})\Bigr) \to \underline{M}\Bigl(\MpZ \otimes (\mbA,\emptyset) \Bigr)
%\biggr), \nonumber
%\end{eqnarray*}
%and know
%\begin{eqnarray*}
%\underline{M}\Bigl(\MpZ \otimes (\mbA,n\{0\})\Bigr) & \simeq & \underline{M}(\MpZ) \otimes \underline{M}(\mbA,n\{0\}), \text{ and }\\
%\underline{M}\Bigl(\MpZ \otimes (\mbA,\emptyset)\Bigr) & \simeq & \underline{M}\MpZ \otimes \underline{M}(\mbA,\emptyset).
%\end{eqnarray*}

%The category $\ulMDM^\eff$ is a tensor triangulated category, so we get the following isomorphism:
%\begin{eqnarray*} 
%{\text{Cone}}(\underline{M}(\OL{M},M^{\infty} + n\ol{Z}) \to \underline{M}\MpM ) \simeq \qquad\qquad\qquad\qquad \\
%\underline{M}\MpZ \otimes ({\text{Cone}}(\underline{M}(\mbA,n\{0\}) \to \underline{M}(\mbA,\emptyset)) \simeq \underline{M}(\OL{Z},Z^{\infty})\otimes Th(N_{\{0\}}\mathbb{A}^1,n) \qedhere
%\end{eqnarray*}

\section{Application} 

Heuristically, the motive with modulus $M(\ol{X}, X^\infty)$ is a place holder which represents the cohomology of $X^\circ = \ol{X} \setminus X^\infty$ whose ramification along the support of $X^\infty$ is bounded by the multiplicities of $X^\infty$. In particular, the case when $X^\infty$ is reduced corresponds to \emph{tamely} ramified cohomology classes. On the other hand, there are concrete connections between tame class field theory and Voevodsky's $\DM^\eff$, cf., the relationship between the tame fundamental group and Suslin homology demonstrated by Geisser, Schmidt, and Spei\ss. In this section we show that these two points of view are compatible.

\begin{thm}\label{redcase}
Let $X$ be a smooth variety over $k$ which has a compactification $\ol{X}$ such that $\ol{X}$ is smooth and $|\ol{X}\backslash X|$ is a strict normal crossings divisor on $\ol{X}$. Then the unit
\[
\ulM(\ol{X},|\ol{X}\backslash X|_{\red}) \to \ul{\omega}^{\eff}({\bf{M}}(X))
\]
of the adjunction $\ul{\omega}_{\eff}:\ulMDM^\eff \rightleftarrows\DM^{\eff}:\ul{\omega}^\eff$ is an isomorphism.
\end{thm}

\begin{lemma}\label{comp}
The functor $\ul{\omega}_\eff$ sends the tame Gysin map $g_{\mf{Z}}\mf{M}$ to Gysin map $g_{Z^{\circ}}M^{\circ}$ of \cite[Thm.3.5.4]{V00b}.
\end{lemma}

\begin{proof}[Proof of Lemma \ref{comp}]
By using excision \cite[Proposition~5.18]{V00c}, Voevodsky’s construction of the Gysin map \cite{V00a} can be restated in terms of the deformation space obtained by blowing up $Z^\circ \times\{0\}$ in $X^\circ \times \PP^1$. The definition of the tame Gysin map is given only by geometrical morphisms, our construction corresponds to Voevodsky's construction under the functor $\ul\omega_\eff$.
\end{proof}

\begin{proof}[Proof of Theorem \ref{redcase}]
%Consider the adjunction $\ul{\omega}_{\eff}:\ulMDM^{\eff}\rightleftarrows\DM^{\eff}:\ul{\omega}^\eff$.  
%@Keiho: Does this sentences actually mean anything?
%Now $\ulMDM^{\eff}$ is compactly generated, so the right adjoint functor $\ul{\omega}^\eff$ is triangulated \cite[Lemma A.11.2]{KSY15}. 
Take
\[
|\ol{X}\backslash X|_{\red} = \Sigma_{i=1}^{n} V_i
\]
where each $V_i$ is an smooth effective Cartier divisor. We prove the claim by induction on $n$.

Let us suppose $n=1$, and write $V$ for $|\ol{X}\backslash X|_{\red} = V_1$. We have the Gysin triangle in $\DM^\eff$ for the closed immersion $V \hookrightarrow \ol{X}$,
\[
{\bf{M}}(\ol{X} \backslash V) \to {\bf{M}}(\ol{X}) \xrightarrow{g_{V}\ol{X}} {\bf{M}}(V)(1)[2] \xrightarrow{+} {\bf{M}}(\ol{X} \backslash V)[1].
\]
%Since $\ul{\omega}^\eff$ is triangulated, 
Since the unit $\id \to \ul{\omega}^\eff\ul{\omega}_{\eff}$ is a natural transformation, we get a morphism of distinguished triangles
\[
\xymatrix{
\ulM(\ol{X},V)  \ar[r] \ar[d]^1 & \ulM(\ol{X},\emptyset) \ar[r]^{} \ar[d]^2 & \ulM({V},\emptyset)(1)[2] \ar[r]^{} \ar[d]^3 & \ulM(\ol{X},V)[1] \ar[d] \\
\ul{\omega}^\eff {\bf{M}}(\ol{X} \backslash V) \ar[r] & \ul{\omega}^\eff {\bf{M}}(\ol{X}) \ar[r] & \ul{\omega}^\eff {\bf{M}}(V)(1)[2] \ar[r]  & \ul{\omega}^\eff {\bf{M}}(\ol{X} \backslash V)[1]
}
\] 
where the vertical arrows are the unit morphisms. Since $\ol{X}$ and $V$ are proper smooth over $k$, $(2)$ and $(3)$ are isomorphisms. Cf.~\cite[Theorem6.3.1]{KMSY20} . So $(1)$ is also an isomorphism.

Now we take 
\[
U = \ol{X} \backslash\bigcup_{i=1}^{n-1}V_i,
\]
and 
\[
W= V_n ~~\backslash(\bigcup_{i=1}^{n-1} V_n\cap V_i).
\]
It is easy to see that $U~\backslash W = \ol{X} ~\backslash \bigcup_{i=1}^{n}V_i$. Now the divisor $\Sigma_{i=1}^{n} V_i$ is a strict normal crossings divisor, so $V_n\cdot_{\ol{X}} V_i = |V_n \cap V_i |_{\red}$. So we get  
\[
\xymatrix{
\ulM(\ol{X},\Sigma_{i=1}^{n} V_i)  \ar[r] \ar[d]^4 & \ulM(\ol{X},\Sigma_{i=1}^{n-1} V_i) \ar[r]^{} \ar[d]^5& \ulM({V_n},\Sigma_{i=1}^{n-1} |V_n\cap V_i|_\red)(1)[2] \ar[r] \ar[d]^6 &  \\
\ul{\omega}^\eff {\bf{M}}(\ol{X} ~\backslash \bigcup_{i=1}^{n}V_i) \ar[r] & \ul{\omega}^\eff {\bf{M}}(U) \ar[r] & \ul{\omega}^\eff {\bf{M}}(W)(1)[2] \ar[r]  & 
}
\] 
By induction, $(5)$ and $(6)$ are isomorphisms. So the claim is proved.
\end{proof}

\section{The case of {$\Z[1/p]$}-coefficients}
%\section{Coefficients away from the characteristic}
%\section{When the characteristic is invertible in the coefficients}

In this section, we suppose that the base field has characteristic $p$. The main objective of this section is to show that the non-Voevodsky part of $\MDM^\eff$ is all $p^\infty$-torsion in the sense that the kernel of $\omega_{\eff} :\MDM^{\eff} \to \DM^{\eff}$ 
is contained in the kernel of $\MDM^{\eff} \to \MDM^{\eff}[1/p]$. 

%understand the motive of $(\PP^1,\{\infty\})\slash(\PP^1,l\{\infty\})$ via Frobenius morphism, and prove Corollary \ref{pinvert}. We shall note a motivation; \[
%\textrm{How will Artin-Schreier theory behave in motives with modulus theory?}
%\]
%Artin-Schreier theory states that etale locally, $\G_{a,X} \overset{\text{Frob} - \id}{\twoheadrightarrow} \G_{a,X}$ is surjective and its kernel is the constant sheaf ${\Z\slash p\Z}_{X}$. To examine Artin-Schreier theory in motives with modulus, we replace $\G_a$ by $(\PP^1,p^c\{\infty\})$, $\text{Frob}:\G_a \to \G_a$ by $\text{Frob}:(\PP^1,p^c\{\infty\}) \to (\PP^1,p^{c-1}\{\infty\})$, $\id:\G_a \to \G_a$ by a natural morphism $(\PP^1,p^c\{\infty\}) \to (\PP^1,p^{c-1}\{\infty\})$. At first we consider Frobenius morphism in $PSh(\ulMCor)$, after that we give a proof of Corollary~\ref{pinvert}.
 
For a natural number $l \in \N$ and an integer $n \in \Z_{\geq 0}$, we define a presheaf $\Zputr(\bcube^{(l/p^{n})}) \in PSh(\ulMCor,\Zpu)$ as
\[
\Zputr(\bcube^{(l/p^{n})}):(\ol{M},\Minf) \mapsto {\ulMCor} \bigr((\ol{M},p^n\Minf),(\mathbb{P}^1,l\{\infty\})\bigl) \otimes \Zpu.
\]

Let us define morphisms 
\begin{eqnarray*}
V^{(n)}:\lpnbcube \to  \lpnobcube,\\
F^{(n)}:\lpnobcube  \to  \lpnbcube,
\end{eqnarray*}
\[
\text{satisfying} \hspace{15pt}  V^{(n)} \circ F^{(n)} = p\cdot id, \hspace{15pt} F^{(n)} \circ V^{(n)} = p\cdot id.
\]

We use the morphism of modulus pairs $\tilde{\pi}:
(\mathbb{P}^1,lp\{\infty\})\to 
(\mathbb{P}^1,l\{\infty\})
$ defined by $k[x] \leftarrow k[x]; x^{p} \mapsfrom x$. Note, this is a minimal morphism which is finite flat on the total space, and therefore has a well defined transpose $\tilde{\pi}^{\text{t}}: 
(\mathbb{P}^1,l\{\infty\}) \to 
(\mathbb{P}^1,lp\{\infty\})
$ as follows.
\begin{lemma}
For a minimal finite flat morphism $g:(\ol{X},X^\infty) \to (\ol{Y},Y^\infty)$, we denote by $g^\circ$ the morphism $\ol{X}\backslash X^\infty \to \ol{Y}\backslash Y^\infty$ given arise to $g$. Then the transpose correspondence ${g^{\circ}}^{\text{t}}\in\Cor(\ol{Y}\backslash Y^\infty,\ol{X}\backslash X^\infty)$ lies in the subgroup $\ulMCor\bigl((\ol{Y},Y^\infty),(\ol{X},X^\infty)\bigr)$.
\end{lemma}

We write $g^t$ for $g^{\circ t}$ considered as a morphism of modulus pairs.

\begin{proof}
It's left proper because $g$ is finite, and admissible because $g$ is minimal.
\end{proof}

\begin{defn}\label{AIAI}
For a modulus pair $(\ol{M},\Minf)$, we define $V^{(n)}(\ol{M},\Minf)$ as the morphism given by
\begin{eqnarray*}
\ulMCor\bigr((\ol{M},p^n\Minf),(\mathbb{P}^1,l\{\infty\})\bigl) &=& \ulMCor\bigr((\ol{M},p^{n+1}\Minf),(\mathbb{P}^1,lp\{\infty\})\bigl)\\ &\xrightarrow{\tilde{\pi} \circ -}& \ulMCor\bigr((\ol{M},p^{n+1}\Minf),(\mathbb{P}^1,l\{\infty\})\bigl)
\end{eqnarray*}
and $F^{(n)}(\ol{M},\Minf)$ as the morphism given by
\begin{eqnarray*}
\ulMCor\bigr((\ol{M},p^{n+1}\Minf),(\mathbb{P}^1,l\{\infty\})\bigl) &\xrightarrow{\tilde{\pi}^{\text{t}} \circ -}& \ulMCor\bigr((\ol{M},p^{n+1}\Minf),(\mathbb{P}^1,lp\{\infty\})\bigl)\\ &=& \ulMCor\bigr((\ol{M},p^{n}\Minf),(\mathbb{P}^1,l\{\infty\})\bigl)
\end{eqnarray*}
\end{defn}
\begin{lemma}\label{vcomp}
$V^{(n)} \circ F^{(n)} = p\cdot id$ and $F^{(n)} \circ V^{(n)} = p\cdot id$.
\end{lemma}

\begin{proof}
We write $\pi$ for the morphism $\mathbb{A}^1 \to \mathbb{A}^1$ given by the morphism of $k$-algebras $k[x] \leftarrow k[x]; x^p \mapsfrom x$. To prove the claim, it is enough to prove that $\pi \circ \pi^{t} = p\cdot id \in \Cor(\mathbb{A}^1,\mathbb{A}^1)$ and $\pi^t \circ \pi = p\cdot id\in \Cor(\mathbb{A}^1,\mathbb{A}^1)$. Since $\pi$ is a flat, finite, surjective morphism with degree $p$, it follows that $\pi \circ \pi^{t} = p\cdot id \in \Cor(\mathbb{A}^1,\mathbb{A}^1)$ is true. So the problem is the other equality.

We need the following lemma.
\begin{lemma}\label{pullback}
The flat pull back $(\text{id}\times\pi)^{*}:Z^{1}(\mathbb{A}^{1}\times\mathbb{A}^{1})\to Z^{1}(\mathbb{A}^{1}\times\mathbb{A}^{1})$ sends $\Gamma_{\pi}$ to $p\cdot\text{id}$. Where $\Gamma_{\pi}$ is the graph of $\pi$, i.e., $\Gamma_{\pi}:\mathbb{A}^{1}\to\mathbb{A}^{1}\times\mathbb{A}^{1}; a \mapsto (a,a^{p})$.
\end{lemma}

\begin{proof}[Proof of Lemma \ref{pullback}]
The ideal of $k[x]\otimes_kk[y]$ corresponding to $\Gamma_\pi$ is $(x^{p}-y)$. Now $\text{id}\times\pi$ comes from the $k$-morphism $k[x]\otimes_k k[y] \to k[x]\otimes_k k[y];  x \mapsto x ,~y\mapsto y^{p}.$ So the pullback of the ideal sheaf $(\text{id}\times\pi)^{*}((x^{p}-y))$ is the ideal sheaf $(x^{p}-y^{p})$, But $\text{ch}(k)=p$, so this is equal to $(x-y)^{p}$. The ideal $(x-y)$ is corresponds to the diagonal morphism $\Delta_{\mathbb{A}^{1}}$, i.e., the identity morphism in in $\Cor(\mathbb{A}^{1},\mathbb{A}^{1})$. 
\end{proof}

Now we recall $\pi$ and $\pi^{t}$ in $\Cor(\mathbb{A}^{1},\mathbb{A}^{1})$. The map $\pi$ is the graph map $\Gamma_{\pi}:\mathbb{A}^{1} \to \mathbb{A}^{1} \times \mathbb{A}^{1}; a \mapsto (a,a^{p})$, and $\pi^{t}$ is the map $\psi:\mathbb{A}^{1}\to\mathbb{A}^{1}\times\mathbb{A}^{1}; b \to (b^{p},b)$. We recall the composition $\pi^{t} \circ \pi$, it is 
\[
\pi^{t} \circ \pi = p_{13*}((\Gamma_{\pi}\times\mathbb{A}^{1})\cdot_{\mathbb{A}^{1}\times\mathbb{A}^{1}\times\mathbb{A}^{1}}(\mathbb{A}^{1}\times\psi)).
\]

Now $\Gamma_{\pi}\times\mathbb{A}^{1}$ and $\mathbb{A}^{1}\times\psi$ are effective Cartier divisors, and they are intersect properly, so 
\[
(\Gamma_{\pi}\times\mathbb{A}^{1})\times_{\mathbb{A}^{1}\times\mathbb{A}^{1}\times\mathbb{A}^{1}}(\mathbb{A}^{1}\times\psi)=(\Gamma_{\pi}\times\mathbb{A}^{1})\cdot_{\mathbb{A}^{1}\times\mathbb{A}^{1}\times\mathbb{A}^{1}}(\mathbb{A}^{1}\times\psi).
\]
Now this is denoted by $V$. Then we have following diagram\scriptsize
\[\xymatrix{
V  \ar@{}[drr]|\square \ar@{^{(}-_>}[rr] \ar@{^{(}-_>}[d]_{} & & \mathbb{A}^{1}\times\mathbb{A}^{1} \ar[rrd]^{} \ar@{^{(}-_>}[d]_{\mathbb{A}^{1}\times\psi} & & \\
\mathbb{A}^{1}\times\mathbb{A}^{1} \ar@{^{(}-_>}[rr]^{\Gamma_\pi \times \mathbb{A}^{1}} \ar@{}[drr]|\square \ar[d] & & \mathbb{A}^{1}\times\mathbb{A}^{1}\times\mathbb{A}^{1} \ar[d]_{p_{12}} \ar[rr]^{p_{13}}   & & \mathbb{A}^{1}\times\mathbb{A}^{1}  \\
\mathbb{A}^{1} \ar[rr]_{\Gamma_\pi} & & \mathbb{A}^{1}\times\mathbb{A}^{1}  & &  
}
\]\normalsize
By definition, we get 
\begin{equation}\label{darkness}
 p_{13} \circ (\mathbb{A}^{1}\times\psi)=\text{id}_{\mathbb{A}^{1}\times\mathbb{A}^{1}}
\end{equation}
and 
\begin{equation}\label{light}
p_{12}\circ(\mathbb{A}^{1}\times\psi) = \text{id}_{\mathbb{A}^{1}}\times\pi.
\end{equation}
The equality (\ref{light}) claims that $V$ is the flat pull back $(id_{\mathbb{A}^{1}} \times\pi)^{*}(\Gamma_{\pi})$, by Lemma \ref{pullback} and \cite[Propostion~7.1]{Fulton} we get
\[
V=p\cdot\text{id}.
\] 
By (\ref{darkness}) we get that $\pi^{t}\circ\pi=p_{13*}(V) = V$. Therefore $\pi^{t}\circ\pi=p\cdot\text{id}$ in $\Cor(\mathbb{A}^{1},\mathbb{A}^{1})$. 
\end{proof}
We consider two colimits $\colim_n (\lpnbcube, V^{(n)})$ and $\colim_n (\ul{\omega}^*\Zputr(\mathbb{A}^1), \ul{\omega}^*\pi)$ in the category $PSh(\ulMCor)$ where the transition maps are $V^{(n)}:\lpnbcube \to \lpnobcube$ and $\ul{\omega}^*\pi:\ul{\omega}^*\Zputr(\mathbb{A}^1) \to \ul{\omega}^*\Zputr(\mathbb{A}^1)$, and the morphism 
\[
I:\colim_n (\lpnbcube, V^{(n)}) \to \colim_n (\ul{\omega}^*\Zputr(\mathbb{A}^1), \ul{\omega}^*\pi)
\]
given by the natural immersions $\lpnbcube\to \ul{\omega}^*\Zputr(\mathbb{A}^1)$. 
\begin{lemma} \label{6.5}There are the following isomorphisms in $PSh(\ulMCor)$.
\begin{align*}
 \Z[1/p]_{\tr}(\bcube^{(l)}) ~&\simeq~ \colim_n (\lpnbcube, V^{(n)}) \\
\ulomega^*\Z[1/p]_{\tr}(\AA^1)~&\simeq ~ \colim_n (\ulomega^*\Zputr(\AA^1),\ulomega^*\pi) 
\end{align*}
\end{lemma}
\begin{proof}
The prime $p$ is invertible in $\Z[1/p]$ so by Lemma \ref{vcomp}, morphisms $V^{(n)}$ are isomorphisms. Similarly the morphism $\pi:\Zputr(\AA^1)\to\Zputr(\AA^1)$ is an isomorphism. So the claim follows.  
\end{proof}
\begin{lemma}\label{III}
$I$ is an isomorphism.
\end{lemma}
\begin{proof}
The problem is surjectivity. By \cite[Lemma 1.1.3]{KMSY19a}
\[
\Cor(\ol{M}\backslash \Minf, \mathbb{A}^1) = \bigcup_{n}\ulMCor((\ol{M},p^{n}\Minf),(\mathbb{P}^{1},l\{\infty\})).
\]
Hence, for any elementary correspondence $W\in \Cor(\ol{M}\backslash \Minf,\mathbb{A}^1)$, there is an integer $n$ such that $W\in \ulMCor((\ol{M},p^{n}\Minf),(\mathbb{P}^{1},l\{\infty\}))$.
\end{proof}

This lemma implies the following theorem which only holds in positive characteristic.

\begin{thm}\label{charap}
For any $l\in \Z_{\geq 1}$, $\ulM((\PP^1,\{\infty\})\slash (\PP^1,l\{\infty\}))\otimes \Zpu=0$.
\end{thm}
\begin{proof}
Since $\Zpu$ is a flat $\Z$-module, it is enough to show that the natural morphism $\Zputr(\PP^1,l\{\infty\}) \to \Zputr(\PP^1,\{\infty\})$ is an isomorphism. There is a commutative diagram
\[
\xymatrix{
\Z[1/p]_{\tr}(\bcube^{(l)}) \ar[r]^-{\simeq} \ar[d] &\colim_n (\lpnbcube, V^{(n)}) \ar[d]^-{I} \\
\ulomega^*\Z[1/p]_{\tr}(\AA^1) \ar[r]^-{\simeq} & \colim_n (\ulomega^*\Zputr(\AA^1),\ulomega^*\pi) 
}
\]
in $PSh(\ulMCor)$, where vertical maps are natural inclusions and horizontal maps are isomophisms given by Lemma~\ref{6.5}. By Lemma~\ref{III} we know that $I$ is an isomorphism. So the natural inclusion $\Z[1/p]_{\tr}(\bcube^{(l)}) \to \ulomega^*\Z[1/p]_{\tr}(\AA^1)$ is also an isomophism for all $l \geq 1$.  The result now follows from the sequence of inclusions $\Z[1/p]_{\tr}(\bcube^{(l)}) \hookrightarrow \Z[1/p]_{\tr}(\bcube) \hookrightarrow\ulomega^*\Z[1/p]_{\tr}(\AA^1)$.
\end{proof}

\begin{cor}\label{Acharap}
For any modulus pair $\MpM$ such that $\ol{M}$ is smooth and $M^\infty_\red$ is strict normal crossing,  \[
\ulM\MpM\otimes\Zpu \simeq \ulM(\ol{M},M^\infty_\red)\otimes\Zpu.
\]
\end{cor}
\begin{proof}
Set $M^\infty=\Sigma_{k=1}^{n} n_k V_k$ where $V_k$ are smooth Cartier divisor. We take $M^\infty_i := n_1 V_1 + \cdots n_i V_i + \Sigma_{k=i+1}^{n} V_k$, it is enough to prove $\ulM(\ol{M},M^\infty_{i})\otimes\Zpu \simeq \ulM(\ol{M},M^\infty_{i-1})\otimes\Zpu$. By Mayer-Vietoris sequence, we can replace $\ulM(\ol{M},M^\infty_{i})\otimes\Zpu$ by $\ulM(\ol{U},\ol{U}\cap M^\infty_{i})\otimes\Zpu$ where $\ol{U}$ has a local chart $q:\ol{U} \to \AA^m$ such that $\ol{U}\cap V_i=q^{-1}(\AA^{m-1}\times\{0\})$ and $\ol{U}\cap(M^\infty_{i}-n_iV_i)=q^{-1}(\{T_1^{d_1}.....T_j^{d_s}=0\})$ where $T_l$ are the coordinates of $\AA^m$. Replace $\ulM(\ol{M},M^\infty_{i})\otimes\Zpu$ by $\ulM(\ol{U},\ol{U}\cap M^\infty_{i})\otimes\Zpu$. In this case we have a diagram $(\Omega)$ used in the proof of the tame Gysin triangle. By Proposition \ref{Fetexcion} the cone of the natural morphisms $\ulM(\ol{M},M^\infty_{i})\otimes\Zpu \to \ulM(\ol{M},M^\infty_{i-1})\otimes\Zpu$ is isomorphic to $\ulM\Bigl((V_i,V_i^\infty) \otimes (\mbA,\{0\}) \Bigr) / \ulM\Bigl((V_i,V_i^\infty) \otimes (\mbA,n_i\{0\})\Bigr)\otimes \Zpu$ where $V_i^\infty=V_i\cdot_{\ol{M}}(n_1 V_1 + \cdots n_{i-1} V_{i-1} + \Sigma_{k=i+1}^{n} V_k)$, Proposition \ref{Fetexcion} and Theorem \ref{charap} claims $(\AA^1,\{0\})\slash (\AA^1,n_i\{0\}) \otimes \Zpu=0$ we win.
\end{proof}

By this corollary and Theorem \ref{redcase}, we get the following theorem.

\begin{thm}[{Corollary~\ref{pinvert}}]\label{ppinv}
If the base field $k$ has characteristic $p$ and admits log resolution of singularities, then there is an equivalence
\[
\omega_{\eff}[1/p]:\MDMeff[1/p] \isom \DM^{\eff}[1/p].
\]
%Indeed there is an equivalence
%\[
%\CI_\Nis[1/p] \simeq  \HI_\Nis[1/p].
%\]
\end{thm}

\begin{proof}
We omit $[1\slash p]$. Since we assume the base field $k$ admits log resolution of singularities, any modulus pair is isomorphic to a modulus pair which has a smooth total space and strictly normal crossing divisor modulus. Now $\MDMeff$ is compactly generated by the $\ulM\MpM$, \cite[Theorem 1(2)]{KMSY19b}, and both $\omega_\eff$ and $\omega^{\eff}$ commute with all sums (the latter because $\omega_\eff$ sends compact generators to compact objects), so it suffices to know that $\ulM\MpM \to \omega_\eff\omega^\eff \ulM\MpM$ is an isomorphism when $\ol{M}$ is smooth and proper and $M^\infty$ is a strict normal crossings divisor. If $M^\infty$ is reduced, Theorem \ref{redcase} implies the claim. By Corollary \ref{Acharap} its also true when $M^\infty$ is not reduced.
\end{proof}

%\begin{proof}
%We show the second claim by using the first claim.

%By \cite[Theorem 2.3.5]{BS18} we have  
%\[
%\MDM^\eff(k) = \{F^\bullet \in D(\MNST)\ |\ H^i(F^\bullet) \in \CI_\Nis\text{ for every } i\in \Z\}.
%\]
%where $\CI_\Nis = CI \cap MNST$, and $MNST \subseteq MPST$ is the full subcategory of those presheaves $F$ such that $\tau_!F \in \underline{M}NST$. 
%By this equation, $\MDM^\eff$ has a t-structure given by the standard t-structure in $D(\MNST)$. To prove the secound isomorphism, it is enough to prove that $\omega_\eff$ preserves t-structures since we know $\omega^{\eff}[1/p]:\MDM^{\eff}[1/p] \isom \DM^{\eff}[1/p]$. There is a commutative square.
%\[
%\xymatrix{
%D(Sh_{\Nis}(\MCor))   & D(Sh_{\Nis}(\Cor))  \ar[l]^{D({\omega}^{\Nis})}  \\ 
%\MDM^{eff} \ar[u]^{j} & \DM^{eff} \ar[l]^{{\omega}^{\eff}} \ar[u]^{j}
%}
%\]
%Since $D(\omega^\Nis)$ is t-exact so $\omega^\eff$ is also t-exact.
%\end{proof}

%\begin{remark}
%We are currently trying to prove \[
%\omega_{\eff}:
%\ulMDM^{\eff}_{prop}(k, \Z / n \Z)
%\isom
%\DM^{\eff}(k, \Z / n \Z).
%\]
%for any positive integer $n\in \Z_{\geq 2}$ when $k$ has characteristic zero. 
%\end{remark}

%\section{Euler realization} 
%In this section, we construct the presheaf $\tilO\in PSh(\ulMCor)$ and prove Theorem \ref{coherent}. 

%\section{Non smooth case}
%
%In this section, we prove a generalization of Corollary~\ref{Acharap} by using B. Bhat and A. Snowden`s refined alteration. After that we prove alteration version of Theorem~\ref{ppinv}. 

\bibliography{Gysin}
\bibliographystyle{alpha}

%%%%%%%%%%%%%%%%%%%%   someday  %%%%%%%%%%%%%%%%%%%%%%%%%%%%%%
%%%%%%%%%%%%%ここから参考文献%%%%%%%%%%%%%%%%%
%

%

\end{document}